\documentclass[11pt]{amsart}

\newtheorem{theorem}{Theorem}[section]

\newtheorem{corollary}[theorem]{Corollary}

\newtheorem{definition}[theorem]{Definition}
\newtheorem{lemma}[theorem]{Lemma}

\newtheorem{proposition}[theorem]{Proposition}

\theoremstyle{definition}
\newtheorem{remark}[theorem]{Remark}
\newtheorem{question}[theorem]{Question}

\usepackage[colorlinks, bookmarks=true]{hyperref}
\usepackage{color,graphicx,shortvrb}

\usepackage{ulem} 
\usepackage{tikz}
\usepackage{comment}
\usepackage{enumerate}
\usepackage{amsmath}
\usepackage{amssymb}

\usepackage[utf8]{inputenc}

\usepackage{cancel}

\def\query#1{\setlength\marginparwidth{80pt}%
\marginpar{\raggedright\fontsize{10}{10}\selectfont\itshape
\hrule\smallskip
{\textcolor{red}{#1}}\par\smallskip\hrule}}

\newcommand{\trm}{\mathcal{T}}

\newcommand{\isom}{\mathrm{ISO}}

\newcommand{\tri}[1]{| \! | \! | {#1} | \! | \! |}
\newcommand{\one}{\mathbf{1}}

\newcommand{\spn}{\mathrm{span}}
\newcommand{\vr}{\varepsilon}

\newcommand{\N}{\mathbb{N}}

\newcommand{\R}{\mathbb{R}}

\newcommand{\mult}{\mathbf{M}}

\newcommand{\norm}{\| \cdot \| }

\newcommand{\triple}[1]{{\left\vert\kern-0.25ex\left\vert\kern-0.25ex\left\vert #1 
    \right\vert\kern-0.25ex\right\vert\kern-0.25ex\right\vert}}

\renewcommand{\emph}[1]{{\it{#1}}}

\begin{document}
\title{Lattice renormings of $C_0(X)$ spaces}

\author{T.~Oikhberg}
\address{Dept. of Mathematics, University of Illinois, Urbana IL 61801, USA}
\email{oikhberg@illinois.edu}

\author{M.A.~Tursi}
\address{Geneva, FL 32732, USA}
\email{maryangelica.tursi@gmail.com}

\maketitle

\begin{abstract}
Suppose $X$ is a locally compact Polish space, and $G$ is a group of lattice isometries of $C_0(X)$ which satisfies certain conditions. 
Then we can equip $C_0(X)$ with an equivalent lattice norm $\tri{ \cdot }$ so that $G$ is the group of lattice isometries of $(C_0(X), \tri{ \cdot })$.
As an application, we show that for any locally compact Polish group $G$ there exists a locally compact  Polish space $X$, and an lattice norm $\tri{ \cdot }$ on $C_0(X)$, so that $G$ is the group of lattice isometries of $(C_0(X), \tri{ \cdot })$.
\end{abstract}

 \section{Introduction}\label{s:displayable}\label{s:intro}

\footnote{To Bill Johnson, friend, mentor, discoverer of many beautiful results}

 In this paper, we strive to renorm a Banach lattice to produce a prescribed group of lattice isometries.
 Recall that a map $T : E \to F$ ($E, F$ are Banach lattices) is called a \emph{lattice isomorphism} if it preserves lattice operations (it suffices to verify $T(x \vee y)  = Tx \vee Ty$ for any $x, y \in E$).
 Further, $T$ is called a \emph{lattice isomorphism} if it is invertible, and both $T$ and $T^{-1}$ are lattice homomorphisms. An isometric  lattice isomorphism is referred to as a \emph{lattice isometry}.
 
 The problem of equipping a Banach space with a new norm with a prescribed group of isometries has a long pedigree. For instance, \cite{Be} shows that any separable Banach space can be renormed to have only the trivial isometries.
 The separability assumption was later removed in \cite{Jar}.
 A more general problem is examined in \cite{FeG_TAMS10} and \cite{Fer11}: given a separable Banach space $Z$, and a group $G$, to equip $Z$ with an equivalent norm $\tri{ \cdot }$ so that $G$ is the group of isometries of $(Z, \tri{ \cdot })$.

Similar problems for Banach lattices have, so far, received less attention. In \cite{OT_AM}, we show that any separable AM-space (in other words, a sublattice of a $C(K)$ space) can be renormed to have only trivial isometries.
In the present paper, we restrict our attention to separable $C_0(X)$ spaces. Our goal is to renorm such a space to obtain a prescribed group of isometries (Theorem \ref{t:main-theorem}, below).

Begin by adapting some definitions from \cite{Fer11} to the lattice setting.
For a Banach lattice $E$, we denote by $\isom(E)$ the group of lattice isometries of $E$ (here and throughout the paper, all isometries are assumed to be surjective).
We equip $\isom(E)$ with the strong operator topology, inherited from $B(E)$.
A \textit{representation} of a topological group $G$ on a Banach lattice $E$ is a continuous group homomorphism $\rho : G \to \isom(E)$; the continuity means that, for any $x\in E$, the map $G \to E : g \mapsto \rho( g)(x)$ is continuous.
A representation is \textit{topologically faithful} if $\rho$ is also a homeomorphism between $G$ and its image $\rho(G)$.  Finally, a \textit{display} of a topological group $G$ is a topologically faithful representation $\rho:G \rightarrow \isom(E)$ such that $E$ can be renormed with an equivalent lattice norm $\triple{\cdot}$ so that $\rho(G) = \isom(E,\triple{\cdot})$.  \\


Before stating the main results of the paper, we need to describe the structure of lattice isometries of spaces of continuous functions.

\begin{proposition}\label{p:isomorphism-characterization}
	Let $T:C_0(X) \rightarrow C_0(Y)$ be a lattice isomorphism, with $X$ and $Y$ locally compact Hausdorff.  Then there exists a homeomorphism $\phi:Y\rightarrow X$ and a continuous map $a:Y\rightarrow (0,\infty)$ with $0 < \inf a \leq \sup a < \infty$, such that for all $f\in C_0(X)$ and $y \in Y$, 
	\[ [Tf] (y) = a(y) \cdot f( \phi(y) ). \]
	Moreover, $a$ and $\phi$ representing $T$ in this fashion are unique.
\end{proposition}

\begin{proof}
By \cite[Theorem 3.2.10]{M-N}, for any lattice homomorphism $T:C_0(X) \rightarrow C_0(Y)$ there exist a continuous map $\phi:Y \supset Y' \rightarrow X$ and a function $a : Y\rightarrow [0,\infty)$, continuous on $Y'$ and vanishing outside of it, so that $[Tf] (y) = a(y) \cdot f( \phi(y) )$ holds for any $f\in C_0(X)$ and $y \in Y$.
As $T$ is bijective, $Y'$ coincides with $Y$, and $\phi : Y \to X$ is bijective. Further, $T^* \delta_y = a(y) \delta_{\phi(y)}$ for any $y \in Y$. As $T^*$ is bounded above and below, so must be $a$.

In a similar fashion, we can write $[T^{-1} g](x) = b(x) g(\psi(x))$, where $b : X \to (0,\infty)$ is continuous with $0 < \inf b \leq \sup b < \infty$, and $\psi : X \to Y$ is a continuous surjection.
It is easy to observe that $T^* \delta_t = a(t) \delta_{\phi(t)}$ for any $t \in Y$, and likewise, $T^{-1*} \delta_s = b(s) \delta_{\psi(s)}$ for any $s \in X$.
The equality $T^{-1*} T^* = I$ implies $\psi = \phi^{-1}$. Consequently, $\phi^{-1}$ is continuous.

Finally, the uniqueness of the representation of $T$ follows from the fact (noted above) that $T^* \delta_y = a(y) \delta_{\phi(y)}$ holds for any $y \in Y$.
\end{proof}

Now suppose $G$ is a group of lattice isometries on $C_0(X)$.
For notational convenience, we will speak of $G$ as acting on $X$, and write $g(t) = \phi_g(t)$ (with $\phi_g$ as in Proposition \ref{p:isomorphism-characterization}).
We shall say that $G$ is \emph{locally equicontinuous} on $X$ if $\{\phi_g\}_{g \in G}$ is equicontinuous on compact subsets of $X$.
Our main result is the following theorem:

\begin{theorem}\label{t:main-theorem}
Suppose $X$ is a locally compact Polish space. Let $G$ be a subgroup of $\isom(C_0(X), \| \cdot \|)$, relatively closed in the strong operator topology. Suppose also that every orbit of $G$ on $X$ is nowhere dense and that $\{\phi_g\}_{g \in G}$ is locally equicontinuous on $X$. Then for all $C > 1$, there exists a renorming $\tri{ \cdot }$ on $C_0(X)$ such that $G = \isom(C_0(X), \tri{ \cdot })$ and $\| \cdot \| \leq \tri{ \cdot } \leq C \| \cdot \|$.
\end{theorem}

Here and below, the \emph{strong operator topology} (\emph{SOT} for short) on $B(E,F)$ is generated by sets $\{T \in B(E,F) : \|T x_0 - y_0\| < \vr$\}, for some $\vr > 0$, $x_0 \in E$, and $y_0 \in F$. This topology corresponds to point-norm convergence of nets of operators.

In general, $\isom(E)$ need not be SOT-closed. The condition of $G$ being \emph{relatively SOT-closed} means that $G$ is closed in the relative strong operator topology of $\isom(C_0(X))$, that is, $G = \overline{G}^{SOT} \cap \isom(C_0(X))$.

If the Banach spaces $E$ and $F$ are separable, then every bounded subset of $B(E,F)$ is SOT-separable. Now suppose $X$ is as in Theorem \ref{t:main-theorem}. For metrizable spaces, separability is equivalent to being second countable (see e.g.~\cite[Corollary 4.1.16]{Engelking}). 
By \cite[Theorem 5.3]{Kechris}, $X$ is a countable union of compact sets. Therefore, $C_0(X)$ is separable, hence all bounded subsets of $B(C_0(X))$ are SOT-separable.

Most of this paper is devoted to proving Theorem \ref{t:main-theorem}.
Throughout, we shall assume that $C \in (1,1.1)$, unless specified otherwise.


In Section \ref{s:generalities} we gather, for later use, some general facts about lattice isomorphisms on $C_0(X)$, and the corresponding maps on $X$. 
In Section \ref{s:combinatorial}, we present a combinatorial construction of the norm $\tri{ \cdot }$, whose existence is claimed in Theorem \ref{t:main-theorem}. The actual properties of the norm are established in Section \ref{s:new norm}.
We present some consequences of Theorem \ref{t:main-theorem} in Section \ref{s:more-results}.
Perhaps the most important result states that any locally compact Polish group is displayable on a $C_0(X)$ space (Corollary \ref{c:everything displayed}).
Finally, in Section \ref{s:isomorphisms}, we consider displayability of bounded groups of lattice isomorphisms (not necessarily isometries).

Throughout this paper, we use standard Banach lattice facts and notation; these can be found in e.g.~\cite{M-N}. 
Our field of scalars is $\R$.

\section{Topological basics}\label{s:generalities}


In this section we analyze families of (locally) equicontinuous mappings on a topological space $X$; the results will be used throughout this paper.

From now on, for any lattice isomorphism $g$ on $C_0(X)$, we let $a_g:X\rightarrow \R$ be the weight function and $\phi_g:X\rightarrow X$ be the underlying homeomorphism as described in Proposition \ref{p:isomorphism-characterization}.
Observe, for later use, that, if $g$ and $h$ are lattice isomorphisms on $C_0(X)$, then $\phi_{gh} = \phi_h \circ \phi_g$, hence, in particular, $\phi_{g^{-1}} = \phi_g^{-1}$. 

\begin{proposition}\label{p:characterize-SOT-convergence}
	Let $(X,d)$ be a locally compact metric space. Suppose $g_n : C_0(X) \to C_0(X)$ ($n \in \N$) are lattice isomorphisms, with $\sup_n \|g_n\| < \infty$. 
	Then $g_n \rightarrow g$ in the SOT if and only if the following conditions hold:
	\begin{enumerate}
	 \item 
	 $\phi_{g_n}|_K \rightarrow \phi_g|_K$ uniformly, for any compact $K \subset X$.
	 \item 
	 $a_{g_n}|_K \rightarrow a_g|_K$ uniformly, for any compact $K \subset X$.
	 \item
	 For any compact $K \subset X$, and any $\varepsilon > 0$, there exists $N \in \N$ so that $d \big(\phi_{g_n}^{-1}(t),\phi_g^{-1}(K)\big) \leq \varepsilon$ for any $t \in K$ and $n \geq N$.
	\end{enumerate}
Moreover, if the family $(\phi^{-1}_{g_n})$ is locally equicontinuous for any compact $K$, then $(1)$ implies $(3)$. Consequently, if $(\phi_{g_n}^{-1})$ is locally equicontinuous, then $g_n \to g$ in SOT if and only if $\phi_{g_n} \to \phi_g$ uniformly on compact sets.
\end{proposition}

For $K \subset X$ and $\delta < 0$, we use the notation $K_\delta$ for $\{t \in X : d(t,K) \leq \delta\}$. The following result may be known.

\begin{lemma}\label{l:compact}
 If $(X,d)$ is locally compact, and $K \subset X$ is compact, then $K_\delta$ is compact for $\delta$ small enough.
\end{lemma}

\begin{proof}
For $t \in X$ and $c > 0$, we let $U(t,c) = \{s \in X : d(t,s) < c\}$.

For each $t \in K$, find $c_t > 0$ so that $\overline{U(t,c_t)}$ is compact. The sets $U(t,c_t)$ ($t \in K$) form an open cover of $K$, which has a finite subcover $(U(t_i, c_{t_i}))_{i=1}^N$. Then $S = \cup_{i=1}^N \overline{U(t_i, c_{t_i})}$ is compact. We claim that there exists $\delta > 0$ so that $K_\delta \subset S$. Indeed, otherwise for any $n \in \N$ there exists $u_n \in K$ and $v_n \in X \backslash S$ so that $d(u_n,v_n) < 1/n$. By passing to a subsequence, we can assume that $u_n \to u \in K$; consequently, $v_n \to u$ as well. 

Find $i$ so that $u \in U(t_i, c_{t_i})$. As the latter set is open, we conclude that $v_n \in U(t_i, c_{t_i})$, for $n$ large enough. However, $v_n \notin S \supset U(t_i, c_{t_i})$, providing us with the desired contradiction.
\end{proof}

\begin{proof}[Proof of Proposition \ref{p:characterize-SOT-convergence}]
Suppose $g_n \to g$ in SOT.
First show that (1) holds. 
Suppose, for the sake of contradiction, that for some compact $K$, there is $\varepsilon > 0$ and a sequence $(k_n) \subseteq K$ such that $d(\phi_{g_n} (k_n), \phi_g(k_n)) > 2\varepsilon$.    By compactness we can assume that $k_n \rightarrow k \in K$.
Now $$d(\phi_{g_n} (k_n), \phi_g(k_n)) \leq d(\phi_{g_n}(k_n), \phi_g(k)) + d(\phi_g(k), \phi_g(k_n)),$$  so for large enough $n$, we have  $d(\phi_{g_n}(k_n), \phi_g(k)) > \varepsilon$.
Pick $f \in C_0(X)$ such that $\|gf\| = 1 = [gf](k) = f(\phi_g(k))$, 
and $f(s) = 0$ when $d(s, \phi_g(k)) \geq \varepsilon$.
Observe that for large enough $n$, $\|g_nf - gf \| < 1/2$, and $k_n$ is such that $[gf](k_n) > \frac{1}{2}$.  In particular, $|[g_nf](k_n) - [gf](k_n) | < \frac{1}{2}$, so $[g_nf](k_n) = a_{g_n} (k_n) f(\phi_{g_n}(k_n)) > 0$.  Since, therefore, $f(\phi_{g_n}(k_n)) > 0$, we must have $d((\phi_{g_n}(k_n), \phi_g(k))) < \varepsilon$, a contradiction.  

To establish (2), 
pick a compact $K \subset X$ and $\delta > 0$. We shall show that, for $n$ large enough, the inequality $|a_{g_n}(t) - a_g(t) | \leq \delta$ holds for any $t \in K$.

Begin by finding $\varepsilon > 0$ so that  $\phi_g (K)_{\varepsilon}$ is compact.
Pick $f \in C_0(X)$ so that $0 \leq f \leq 1$, and $f = 1$ on $\phi_g (K)_{\varepsilon}$. Select $N \in \N$ so that $\|g_n f - gf \| < \delta$ for $n \geq N$.

Above, we have established that $\phi_{g_n} \to \phi_g$ uniformly on compact sets; therefore, there exists $M \in \N$ so that $\phi_{g_n}(K) \subset \phi_g (K)_{\varepsilon}$ for $n \geq M$.
For $n \geq \max\{M,N\}$, and for $t \in K$, we have
$$ |a_{g_n}(t) - a_g(t) | = |a_{g_n}(t) f(\phi_{g_n}(t) ) - a_g(t)f(\phi_g(t)) | \leq \|g_n f - gf \| < \delta . $$ 

To handle (3), suppose, for the sake of contradiction, that for some compact $K$ and $\varepsilon > 0$, there exist $n_1 < n_2 < \ldots$ and $(t_i) \subset X \backslash \phi^{-1}(K)_\varepsilon$ so that $\phi_{n_i}(t_i) \in K$.
Passing to a subsequence if necessary, we can and do assume that there exists $\delta \in (0,\varepsilon)$ so that $d(\phi_g(t_i),K) > \delta$.
Indeed, otherwise $\liminf_i d(\phi_g(t_i),K) = 0$, hence, passing to a subsequence and invoking the compactness of $K$, we can assume that $\phi_g(t_i) \to s \in K$. Consequently, $t_i \to \phi_g^{-1}(s) \in \phi_g^{-1}(K)$, which is impossible.

Find $f \in C_0(K)$ so that $0 \leq f \leq 1$, $f|_K = 1$, and $f|_{X \backslash K_\delta} = 0$.
Then $[gf](t_i) = a(t_i) f ( \phi_g(t_i))  = 0$, while $[g_{n_i} f](t_i) = a(t_i) f ( \phi_{g_{n_i}}(t_i))$, and so,
$$\big\|gf - g_{n_i} f\big\| \geq \big| [gf](t_i) - [g_{n_i} f](t_i) \big| \geq \inf a > 0 . $$
This precludes $g$ from being the SOT limit of $(g_n)$.

Suppose (1), (2), and (3) hold, and
show that $g_n \to g$ in SOT. 
By the boundedness of $(g_n)$ it suffices to show that, for any norm one $f \in C_0(X)$, supported on a compact set $K$, we have $\|g_n f - gf\| \to 0$.

By Lemma \ref{l:compact}, we can fix $\varepsilon > 0$ for which $K_\varepsilon$ and $\phi^{-1}_g(K)_{\varepsilon}$ are compact.
Further, find $\delta > 0$ so that $|f(t) - f(s)| < \varepsilon$ whenever $d(t,s) < \delta$ (this is possible, by the uniform continuity of $f$).
Find $N \in \N$ so that, for any $n \geq N$, the following conditions hold: (i) $d(\phi_{g_n}(t), \phi_g(t)) \leq \delta$ for any $t \in \phi^{-1}_g(K)_{\varepsilon}$, 
(ii) $|a_{g_n}(t) - a(t)| < \varepsilon$ for any $t \in K_\varepsilon$, (iii) $\phi_{g_n}^{-1}(K) \subset \phi_g^{-1}(K)_{\varepsilon}$.
For $t \notin \phi_g^{-1}(K)_\varepsilon$, neither $\phi_g(t)$ nor $\phi_{g_n}(t)$ belongs to $K$, and therefore, $[g_n f](t) = 0 = [gf](t)$.
On the other hand, for $t \in \phi_g^{-1}(K)_\varepsilon$, the triangle inequality yields
\begin{align*}
&
\big| [g_n f](t) - [gf](t) \big| =
\big| a_{g_n}(t) f(\phi_{g_n}(t)) - a_g(t) f(\phi_g(t)) \big| 
\\
&
\leq
\big| a_{g_n}(t) \big| \cdot \big| f(\phi_{g_n}(t)) - f(\phi_g(t)) \big| +
\big| a_{g_n}(t) - a_g(t) \big| \cdot \big| f(\phi_g(t)) \big| 
\\
&
\leq
C \big| f(\phi_{g_n}(t)) - f(\phi_g(t)) \big| +
\big| a_{g_n}(t) - a_g(t) \big| \cdot \big| f(\phi_g(t)) \big| ,
\end{align*}
where $C = \sup_n \|g_n\| = \sup_n \|a_{g_n}\|_\infty$. 
For $n \geq N$, we have $\big| a_{g_n}(t) - a_g(t) \big| \leq \varepsilon$, and $d(\phi_{g_n}(t)), \phi_g(t)) < \delta$, hence $\big| f(\phi_{g_n}(t)) - f(\phi_g(t)) \big| \leq \varepsilon$.
Therefore, for any $n \geq N$ we have
$$
\|g_nf - gf\| = \sup_t \big| [g_n f](t) - [gf](t) \big| \leq (C+1) \varepsilon .
$$
Since $\varepsilon$ can be taken arbitrarily small, we conclude that $\lim_n \|g_nf - gf\| = 0$.

It remains to establish the ``moreover'' statement. Suppose, for the sake of contradiction, that the family $(\phi^{-1}_{g_n}|_S)$ is equicontinuous for any compact $S$, and $(1)$ holds while $(3)$ doesn't. 
Find a compact $K \subset X$, $\varepsilon > 0$, $n_1 < n_2 < \ldots$, and $(t_i) \subset X \backslash \phi^{-1}(K)_\varepsilon$ so that $\phi_{g_{n_i}}(t_i) \in K$. Passing to a subsequence, assume that $\phi_{g_{n_i}}(t_i) \to s \in K$.
Find $\sigma > 0$ so that $d(\phi_{g_n}^{-1}(t), \phi_{g_n}^{-1}(s)) < \varepsilon$ whenever $n \in \N$, and $t,s \in K$ satisfy $d(t,s) < \sigma$.
Further, find $\rho > 0$ so that $K_\rho$ is compact.

As $\phi_g^{-1}(s) \in \phi_g^{-1}(K)$, $d \big(t_i, \phi_g^{-1}(s)\big) > \varepsilon$. Moreover, $\lim_i \phi_{g_{n_i}} = \phi_g$ uniformly on compact sets, hence
$$
\lim_i \phi_{g_{n_i}} \circ  \phi_g^{-1}(s) = \phi_g \circ  \phi_g^{-1}(s) = s .
$$
Thus, for $i$ large enough, $\phi_{g_{n_i}} \circ  \phi_g^{-1}(s) \in K_\rho$.
As the family $\phi^{-1}_{g_n}$ is equicontinuous on $K_\rho$, we have
$$
\lim_i d \big( \phi_{g_{n_i}}^{-1} \circ  \phi_{g_{n_i}} \circ  \phi_g^{-1}(s) , \phi^{-1}_{g_{n_i}}(s) \big) = \lim_i d \big( \phi_g^{-1}(s) , \phi^{-1}_{g_{n_i}}(s) \big) = 0
$$
As $d \big( t_i, \phi_g^{-1}(s) \big) \geq d(t_i,\phi^{-1}(K)) > \varepsilon$, we conclude that $d \big(t_i, \phi_{g_{n_i}}^{-1}(s)\big) > \varepsilon$ for $i$ large enough.
For such $i$, we must have $d \big(\phi_{g_{n_i}}(t_i), s \big) \geq \sigma$, which contradicts $\phi_{g_{n_i}}(t_i) \to s$.
\end{proof}

\begin{remark}\label{r:equicontinuity required}
 In general, items (1) and (2) of Proposition \ref{p:characterize-SOT-convergence} alone do not imply that $g_n \to g$ in SOT. Below we construct an example of a locally compact metrizable space $X$, and a sequence of continuous maps $\phi_1, \phi_2, \ldots : X \to X$, with continuous inverses, so that $\phi_n \to id$ uniformly on compact sets, but there exists $x \in C_0(X)$ so that $x \circ \phi_n \not\to x$.
 
 Let $X$ consist of pairs $(0,x)$ ($x \in \N \cup \{\infty\}$) and $(i,j)$, with $i,j \in \N$.
 Define the metric $d$ by setting $d \big( (0,x), (0,y) \big) = 2^{-\min\{x,y\}}$, and $d \big( (a,b), (c,d) \big) = 1$ if $\max\{a,c\} \geq 1$.
 Then $K \subset X$ is compact iff it is a union of a finite set and an interval $\{(0,x) : n \leq x \leq \infty\}$ ($n \in \N$).
 
 For $n \in \N$ define $\phi_n : X \to X$ as follows:
 \begin{itemize}
  \item $\phi_n(0,i) = (0,i+1)$ for $n \leq i < \infty$, $\phi_n(0,\infty) = (0,\infty)$.
  \item $\phi_n(n,n) = (0,n)$.
  \item $\phi_n(n,i) = (n,i-1)$ for $n < i < \infty$.
  \item $\phi_n(a,b) = (a,b)$ if either $a \notin \{0,n\}$, or $a \in \{0,n\}$ and $b < n$. 
 \end{itemize}
 Clearly, for each $n$, $\phi_n$ is continuous, bijective, and has a continuous inverse. Further, $\phi_n \to id$ uniformly on compact sets. However, the family $(\phi_n^{-1})$ is not equicontinuous on $\{(0,i) : i \geq m\}$, for any $m$.
 
 To show that the lattice isomorphisms $g_n : x \mapsto x \circ \phi_n$ do not converge to the identity in SOT, define $x \in C_0(X)$ by setting $x(a,b) = 1$ if $a = 0$, $x(a,b) = 0$ otherwise. Then $x \circ \phi_n(n,n) = x(0,n) = 1$, and so, $x \circ \phi_n \not\to x$ in norm.
\end{remark}

From this point on, and up to the end of Section \ref{s:more-results}, $G$ denotes a group of lattice isometries of $C_0(X)$.
The above discussion shows that for any $g \in G$ there exists a (necessarily unique) homeomorphism $\phi_g : X \to X$ so that $gx = x \circ \phi_g$ (in other words, $a_g = 1$).
For ease of notation, we sometimes use $g(t)$ for $t\in X$ to refer to $\phi_g(t)$.

Finally, the following description of SOT convergence will be used in the sequel.

\begin{lemma}\label{l:equicontinuous+SOT=isometry}
	Let $X$ be a locally compact Polish space, and suppose  $\{\phi_g\}_{g \in G}$ is locally equicontinuous on $X$.  Then $g_n \rightarrow g$ in the SOT on $C_0(X)$ iff $\phi_{g_n} \rightarrow \phi_g$ pointwise on $X$.  
\end{lemma}

\begin{proof}
	By Proposition \ref{p:characterize-SOT-convergence}, convergence in SOT clearly implies pointwise convergence.
	
 	Suppose now that $g_n(t) \rightarrow g(t)$ for all $t\in X$. 
 	Since $G$ is closed under inverses and $\phi_g^{-1} = \phi_{g^{-1}}$ for all $g\in G$, $(\phi_{g_n}^{-1})_n$ is also locally equicontinuous, by Proposition \ref{p:characterize-SOT-convergence},  we need only to show that $\phi_{g_n} \to \phi_g$ uniformly on compact sets. Fix $\varepsilon > 0$ and a compact $K \subset X$.
 	We have to show that, for $n$ large enough,
 	$\sup_{t \in K} d \big( \phi_{g_n}(t), \phi_g(t) \big) < \varepsilon$.
 	Find $\delta > 0$ so that $d \big( \phi_{g_n}(t), \phi_{g_n}(s) \big) < \varepsilon/3$ whenever $n \in \N$, $t,s \in K$, and $d(t,s) < \delta$.
 	By the compactness of $K$, find $t_1, \ldots, t_m$ so that for any $s \in K$ there exists $i$ with $d(s,t_i) < \delta$.
 	Find $N \in \N$ so that
 	$$  \sup_{n \geq N} \max_{1 \leq i \leq m} d \big( \phi_{g_n}(t_i) , \phi_g(t_i) \big) < \frac{\varepsilon}3 .  $$
 	Now fix $n \geq N$ and $t \in K$. Find $i$ so that $d(t,t_i) < \delta$. Then
 	\begin{align*}
 	 &
 	d \big( \phi_{g_n}(t), \phi_g(t) \big) \leq
 	\\
 	&
 	d \big( \phi_{g_n}(t), \phi_{g_n}(t_i) \big) + d \big( \phi_{g_n}(t_i), \phi_g(t_i) \big) + d \big( \phi_g(t_i), \phi_g(t) \big) < 3 \cdot \frac\varepsilon3 = \varepsilon . \qedhere 
    \end{align*}

 	%
\end{proof}

 

	\section{Proof of Theorem \ref{t:main-theorem}: a combinatorial construction}\label{s:combinatorial}
	
In this section we lay the combinatorial groundwork required to prove Theorem \ref{t:main-theorem}. Based on this, in Section \ref{s:new norm} we define a new norm $\tri{ \cdot }$ on $C_0(X)$, and show it has the desired properties.
	Throughout we use the notations and assumptions of Theorem \ref{t:main-theorem}.

In a natural way, we view $G$ as acting on $X^n$, for any $n$: for $  t = (t_1, \ldots, t_n)$ (with $t_i \in X$ for $1 \leq i \leq n$), we define $g  t := (gt_1, \ldots, gt_n)$. Equipping $X^n$ with the max metric $d_n$ (which generates the product topology), we see that the action of $G$ on $X^n$ is locally equicontinuous.

For $  t \in X^n$, we shall denote by $[  t]$ the closure of the $G$-orbit of $  t$. For $  s,   t \in X^n$, we write $  s \sim   t$ if $[   s] = [  t]$. For future use, we need a criterion for this equivalence:

\begin{lemma}\label{l:equivalences}
    In the above notation, $  s \sim   t$ if and only if $  s \in [  t]$.
\end{lemma}

\begin{proof}
If $  s \sim   t$, then $  s \in [  t ]$ due to $[  t]$ being closed.
    Now suppose, conversely, that $  s \in [  t]$ -- that is, there exists a sequence $(g_i)_i \subset G$ so that $g_i   t \to   s$. Then, for any $h \in G$, the continuity of the action of $h$ implies $h g_i   t \to h   s$, hence the orbit of $  s$ belongs to $[  t]$, and consequently, $[  s] \subset [  t]$.
    To obtain the converse inclusion, it suffices to show that $  t \in [  s]$. Fix $\vr > 0$, and show that $d_n(g_i^{-1}  s,   t) < \vr$ for $i$ large enough. Let $U$ be an open neighborhood of $  t$, so that $\overline{U}$ is compact. Find $\delta > 0$ so that $d_n(g   x,g   y) < \vr$ if $  x,  y \in U$ satisfy $d_n(  x,  y) < \delta$. The desired conclusion now follows from the fact that, for $i$ large enough, $d_n(g_i  t,  s) < \delta$, and $g_i   t \in U$.
\end{proof}

Note that if every orbit of $G$ in $X$ is nowhere dense, then $X$ cannot have any isolated points, so $C_0(X)$ is atomless.

\begin{lemma}\label{l:dense}
We can find a dense sequence $t_1^0, t_2^0, \ldots \in X$ so that the orbits of these points have disjoint closures.
\end{lemma}

\begin{proof}
    Let $(s_i)_i$ be a dense subset of $X$. Let $t_1 := s_1$. Beyond that, proceed recursively: suppose we have already defined $(t_i^0)_{i=1}^k$ so that, for every such $i$, $d(s_i, t_i) < 2^{-i}$, and the orbits of points $t_i$ ($1 \leq i \leq k$) have disjoint closures. As such orbits are nowhere dense, we can find $t_{k+1} \in X \backslash (\cup_{i=1}^k [t_i^0])$ so that $d(s_{k+1}, t_{k+1}) < 2^{-(k+1)}$. Clearly the sequence $(t_i)_i$ obtained in this manner is dense in $X$.
\end{proof}


We now introduce some notation which will be used throughout the remainder of the proof of this theorem.  Let $G'$ be a countable SOT-dense (equivalently, in light of Proposition \ref{p:characterize-SOT-convergence}, dense in the topology of uniform convergence on compact subsets of $X$) subgroup of $G$. 
Fix a dense sequence $(t_i^0) \subset X$, so that the $G$-orbits of the points $t_i^0$ have disjoint closures (such a sequence exists, by Lemma \ref{l:dense}). Let $(t_i^\gamma)_{\gamma=0 }^{\gamma <|G't_i^0|}$ be a sequence enumerating the elements of the $G'$-orbit of $t_i^0$.
Let  
\[ N := \{ (t_i^{\gamma_0}, t_{i+1}^{\gamma_1}, ... , t_{i+n}^{\gamma_n} ): i \in \N, n > 1, 0\leq \gamma_j < | G' t^0_{i+j} | \} \]

 We say that $  s= (s_0,...,s_n),   t = (t_0,...,t_n) \in X^{n+1}$ ($n \in \N$) are \textit{comparable} if $[s_i] = [t_i]$ for all $0 \leq i \leq n$.
Define $$\widetilde{N} = \{   s\in X^{<\infty}: s \text{ comparable to some } t\in N \} ,$$ and let $\iota: \widetilde{N}: \rightarrow \N^{< \infty}$ be the function taking $s\in \widetilde{N}$ to its underlying subscript indices: if $s$ is comparable to $\big(t_i^{\gamma_0}, \ldots, t_{i+n}^{\gamma_n} \big)$, then $\iota(s) = (i, \ldots, i+n)$.
Clearly, any $s,t\in \widetilde{N}$ are comparable iff $\iota(s) = \iota(t)$ (and, implicitly, $|s| = |t|$; here $|s|$ is the number of elements in the sequence $s$). \\

\medskip


Throughout this paper, we work extensively with elements of $X^{n+1}$ -- that is, with $(n+1)$-tuples of points of $X$.
In general, such an tuple $t$ will be written as $t = (t_0, t_1, \ldots, t_n)$. However, if the exact position of $t$ in $N$ needs to be emphasized, we shall write $t = (t_i^{\gamma_0},..., t_{i+n}^{\gamma_n})$, with $t_{i+r}^{\gamma_r} \in [t_{i+r}^0]$ defined above.

Another common notation describes the ``restriction'' of a tuple. 
For $t = (t_i, \ldots, t_{i+n}) \in X^{n+1}$ ($i \geq 0, n \in \N$), and $0 \leq a \leq b \leq n$, we denote by $t|_a^b$ to be the segment $(t_{i+a}, t_{i+a+1}..., t_{i+b-1}, t_{i+b})$; if $a=0$, we just write $t|^b$. \\

The same convention is applied to infinite sequences. 
We denote by $[[N]]$ the set of sequences $\tau = (\tau_j)_{j=0}^\infty = (t_i^{\gamma_0}, t_{i+1}^{\gamma_1}, ... )$. For such $\tau$ let $\tau|^b = (\tau_j)_{j=0}^b = (t_i^{\gamma_0}, t_{i+1}^{\gamma_1}, ... t_{i+b}^{\gamma_b})$.


The following combinatorial lemma will be essential to establishing Theorem \ref{t:main-theorem}.

\begin{lemma}\label{l:bmap}
Suppose a sequence $(\lambda_i)_i \subset (1,C)$ decreases to $1$. 
	There exist $L > 9$ and maps $b: N\rightarrow \R$, $c:\widetilde{N} \rightarrow 3\N$ with the following properties: 
	\begin{enumerate}
\item\label{b encodes} For all $t,s \in N$, $ b(t) = b(s) \iff gt = s$ for some $g\in G$.

\item\label{c encodes} For all $t,s \in \widetilde{N}$, $c(t) = c(s)$ $\iff$ $t$ and $s$ are comparable.

\item\label{large L} $\sum_{n=1}^\infty L^{-n} < C - \lambda_1$.

\item\label{value of b} For any $t \in N$, $b(t)\in [L^{c(t)-1}, L^{c(t)}]$,
  and the set $\{b(s) : s \in N, \iota(s) = \iota(t)  \}$ forms a strictly increasing (finite or infinite) sequence with supremum $L^{c(t)}$.

 \item\label{est on b}
 For $t = \big(t_i^{\gamma_0}, \ldots, t_{i+n}^{\gamma_n} \big)$, we have $b(t) \geq L^{3(i+n)-4}$.
 
 \item\label{b increases}
 For $t = \big(t_i^{\gamma_0}, \ldots, t_{i+n}^{\gamma_n} \big)$ and $t' = \big(t_i^{\gamma_0}, \ldots, t_{i+n}^{\gamma_n}, t_{i+n+1}^{\gamma_{n+1}} \big)$, we have $b(t') > L b(t)$.

\item\label{i:long sum} For any $\tau \in [[N]]$, with the initial element $\tau_0 = t_i^{\gamma_0}$, we have $$\lambda_i + \sum_{m=1}^\infty \frac{1}{b(\tau|^m)} < C . $$
 
	\end{enumerate}
\end{lemma}	

\begin{proof}
First, construct an enumeration $(a_m)_m$ of all finite consecutive sequences of positive integers of length at least 2 as follows:
\begin{align*}
a_1 &= (1,2) \\
a_{k+1} & = (i-1, i, i+1, ..., i+n) \text{ if }a_k = ( i, i+1, ..., i+n), i > 1\\
a_{k+1} & = (n,n+1)  \text{ if }a_k = ( 1, 2, ..., n)
\end{align*}

We can lay out $(a_m)_m$ in a triangular format as follows:
\begin{align*}
&(1,2)  		 \qquad & \qquad & \\
&(2,3), \quad \ \ (1,2,3), \qquad & \qquad &\\
&(3,4),\quad \ \ (2,3,4),\qquad \quad \ \ (1,2,3,4), & \qquad  & \\
& \vdots \\
&(i,i+1),\ (i-1,i,i+1),\ \ (i-2,i-1,i,i+1), \ \  \dots, \ \ (1,...,i+1)
\end{align*}
Our enumeration proceeds left to right in each row, and then jumps to the start of the next row. Note that tuples in the same column have the same length, and tuples in the same ``North-West to South-East'' diagonal begin with the same element.  

Given $m \in \N$, let \[ N_m = \{ t\in \widetilde{N}: \iota(t) = a_m  \}\]
For $t \in N_m$ let $c(t) :=c_m = 3m$. As this $c$ encodes comparability, it satisfies item 2 of the lemma.
Partition $N_m$ into $\sim$-equivalence classes $S_m^1,S_m^2,...$ -- that is, $t$ and $s$ belong to the same $S_m^i$ iff $[t] = [s]$, and $S_m^i \cap S_m^j = \emptyset$ if $i \neq j$. 
For the sequence $(S_m^i)_i$, find a strictly increasing sequence $(k_m^i)_i \subseteq [c_m - 1, c_m]$, with $\lim_i k_m^i = c_m$ (if the sequence is finite -- that is, if $N_m$ splits into finitely many equivalence classes -- pick $c_m$ for the last point). For any $t \in S_m^i$, let $b'(t) = k_m^i$, then $b'(t) = b'(s)$  iff $[t] = [s]$. It follows that Properties \ref{b encodes} and \ref{c encodes} hold.

Pick $L > 9$ such that $\sum_{n=1}^\infty L^{-n} < C - \lambda_1$ (that is \ref{large L} holds), and let $b(t) = L^{b'(t))}$.
By our choice of $b'(t) = k_m^i$, Property \ref{value of b} also holds. 

To establish Property \ref{b increases}, suppose $a_m = (i, \ldots, i+n)$ and $a_\ell = (i, \ldots, i+n, i+n+1)$. Then $a_\ell$ is located in the row below $a_m$, hence $\ell > m$. By the above, $b \big(t_i^{\gamma_0}, \ldots, t_{i+n}^{\gamma_n}\big) \leq L^{c_m}$, and $b \big(t_i^{\gamma_0}, \ldots, t_{i+n}^{\gamma_n}, t_{i+n+1}^{\gamma_{n+1}}\big) \geq L^{c_\ell-1}$. Now recall that $c_\ell - c_m \geq 3$.

Property \ref{i:long sum} directly follows from Properties \ref{value of b} and \ref{large L}.

To establish Property \ref{est on b}, suppose $a_m = (i, \ldots, i+n)$. Then $a_m$ sits in row $i+n-1$ of our table.
The preceding $i+n-2$ rows contain in total $1 + 2 + \ldots + (i+n-2) \geq i+n-2$ elements, hence $m \geq i+n-1$, so $c_m \geq 3(i+n-1)$. Now recall that $b \big(t_i^{\gamma_0}, \ldots, t_{i+n}^{\gamma_n}\big) \geq L^{c_m - 1}$.
\end{proof}

\begin{remark}
    The enumeration described above has an additional interesting property: if $a_m = (i, \ldots, i+n)$ and $a_\ell = (p, \ldots, p+q)$, with $q \leq n$ and $j \leq i$, then $\ell \leq m$.
\end{remark}

\section{Proof of Theorem \ref{t:main-theorem}: the new norm and its properties}\label{s:new norm}

We now define an equivalent norm on $C_0(X)$:
\begin{equation}
\triple{x} = \sup_{t \in N} \rho_t(x),
\label{eq:define_new_norm}
\end{equation}
where, for $t = \big( t_i^{\gamma_0} , \ldots, t_{i+n}^{\gamma_n} \big) \in N$, the seminorm $\rho_t( \cdot )$ is defined by

\begin{equation}
\rho_t(x) = \lambda_i |x(t_i^{\gamma_0})| + \sum_{k=1}^{n} \frac{| x(t_{k+i}^{\gamma_k})|}{b(t|^k)} .
\label{eq:invariance}
\end{equation}
Our goal is to show that this norm has the properties outlined in Theorem \ref{t:main-theorem}. Observe first an invariance property of the seminorms from \eqref{eq:invariance}.

\begin{lemma}\label{l:rho-invariance}
For any $t \in N$, $x \in C_0(X)$, and $g \in G'$, we have $\rho_{gt}(x) = \rho_t(gx)$.
\end{lemma}

Recall that we identify the action of group $G$ on $C_0(X)$ with the induced action on $X$, that is, $[gx](k) = x(gk)$.

\begin{proof}
From the definition of $\rho_t$,
$$
\rho_t(gx) = \lambda_i |x(gt_i^{\gamma_0})| + \sum_{k=1}^{n} \frac{| x(gt_{k+i}^{\gamma_k})|}{b(t|^{k})} .
$$
By the definition of $b( \cdot )$ given in Section \ref{s:combinatorial}, $b(gs) = b(s)$ for any $s\in N$. Consequently, the right hand side of the centered expression above equals
$$
\lambda_i |x(gt_i^{\gamma_0})| + \sum_{k=1}^{n} \frac{| x(gt_{k+i}^{\gamma_k})|}{b(gt|^{k})} = \rho_{gt}(x) . \qedhere
$$
\end{proof}

For future use, we gather some simple facts about $\triple{ \cdot }$.

\begin{lemma}\label{l:simple_properties} 
 For any $x \in C_0(X)$, we have $\|x\| \leq \triple{x} \leq C \|x\|$.
  Furthermore, the inequality $\tri{x} \geq \lambda_i |x(t)|$ holds for any $t \in [t_i^0]$.
\end{lemma}

\begin{proof}
 Fix $x \in C_0(X)$ and $k \in X$. By picking $t_i^{\gamma_0}$ arbitrarily close to $k$, we have $|x(k)| \leq \triple{x}$. Consequently, $\|x\| \leq \tri{x}$. 
On the other hand, Lemma \ref{l:bmap}(\ref{i:long sum}) shows that $\rho_t(x) \leq C \|x\|$ for any $x \in C_0(X)$ and $t \in N$.

 To handle the ``furthermore'' statement, suppose first that $t = t_i^\gamma$ for some $\gamma$. Set $s = (t_i^\gamma, t_{i+1}^0)$, and recall that $\tri{x} \geq \rho_s(x) \geq \lambda_i |x(t)|$.
For an arbitrary $t \in [t_i^0]$, use the continuity of $x$, and approximation of $t$ by $t_i^\gamma$'s.
\end{proof}

Our goal is to show that $G = \isom(C_0(X), \tri{\cdot})$. The inclusion in one direction is easy to establish.

\begin{lemma}\label{l:G_invariance}
	In the above notation, $G \subseteq \isom(C_0(X), \tri{\cdot})$.
\end{lemma}

\begin{proof}
	It suffices to show that, for any $x \in C_0(X)$, any $g \in G$, any $\varepsilon > 0$, and for any $t = (t_0, \ldots, t_n) \in N$ we have $\tri{gx} \geq \rho_t(x) - \varepsilon$.
	Recall that $G'$ is dense in $G$ in the SOT topology (equivalently, in the topology of uniform convergence on $X$), hence we can find $h \in G'$ so that $|x(t_i) - x(h^{-1}gt_i)| \leq \varepsilon/(C(n+1))$ for $0 \leq i \leq n$. By Lemma \ref{l:rho-invariance},
	$$
	\tri{gx} \geq \rho_{h^{-1}t}(gx) \geq \rho_{h^{-1}t}(hx) - C \sum_i |x(t_i) - x(h^{-1}gt_i)| \geq \rho_t(x) - \varepsilon ,
	$$
	which is what we need.
\end{proof}

The proof of Theorem \ref{t:main-theorem} relies on computing norms of finite linear combinations of atoms in $(C_0(K), \tri{ \cdot })^*$. We therefore need the following consequence of Lemma \ref{l:G_invariance}. 
Here and below, $\tri{ \cdot }^*$ denotes the norm on the space of finite signed Radon measures on $X$, dual to $\tri{ \cdot }$.

\begin{lemma}\label{l:limits of atoms}
Suppose $t = (t_1, \ldots, t_n), s = (s_1, \ldots, s_n) \in X^n$ are such that $[t] = [s]$. Then, for any $\beta_1, \ldots, \beta_n \in \R$, we have
$\tri{ \sum_{k=1}^n \beta_k \delta_{t_k} }^* = \tri{ \sum_{k=1}^n \beta_k \delta_{s_k} }^*$.
\end{lemma}

\begin{proof}
Find a sequence $(g_j) \subset G$ so that $(s_1^{(j)}, \ldots, s_n^{(j)}) := g_j t \to s$. Then
 $$
 {\mathrm{weak}}^*-\lim_j \sum_{k=1}^n \beta_k \delta_{s_k^{(j)}} = \sum_{k=1}^n \beta_k \delta_{s_k} .
 $$
 Therefore, taking the $G$-invariance of $\tri{ \cdot }$ (and $\tri{ \cdot }^*$) into account, we obtain
 $$
 \tri{ \sum_{k=1}^n \beta_k \delta_{s_k} } \leq \liminf_j \tri{ \sum_{k=1}^n \beta_k \delta_{s_k^{(j)}} } = \tri{ \sum_{k=1}^n \beta_k \delta_{t_k} } .
 $$
 
 By the local equicontinuity of the action of $G$, we have $g_j^{-1} s \to t$. The same reasoning as above yields
 $$
 \tri{ \sum_{k=1}^n \beta_k \delta_{s_k} } \geq \tri{ \sum_{k=1}^n \beta_k \delta_{t_k} } .
 \qedhere
 $$
\end{proof}

To complete the proof of Theorem \ref{t:main-theorem}, we shall show that any isometry in $(C_0(X), \tri{\cdot})$ must also be in $G$.  To this end, we use the following series of lemmas that characterize the dual norm $\triple{\cdot}$.  


\begin{lemma}\label{l:unit}
	Suppose $t\in X$ with $\|t\| = 1$. If $t \in [t_i^0]$ for some $i$, then $\triple{\lambda_i \delta_t}^*= 1$. If $t \notin \cup_i [t_i^0]$, then $\triple{\delta_t}^* = 1$.
\end{lemma}

\begin{proof}
Suppose first $t \in [t_i^0]$.
As noted before, for every $x \in C_0(X)$ we have $\lambda_i |x(t)| \leq \tri{x}$,
hence $\tri{ \lambda_i \delta_t}^*	\leq 1$.  

To prove the opposite inequality, fix $\varepsilon > 0$, and find $x \in C_0(X)_+$ so that $\rho_s(x) \leq 1+\vr$ for any $s = \big( s_p^{\kappa_0}, \ldots, s_{p+q}^{\kappa_q} \big) \in N$ ($p, q \geq 1$), and $\lambda_i x(t) = 1$.
To this end, find $M$ so large that $\lambda_M + L^{5-3M} < 1 + \vr$.
Find an open set $U \subset X$ containing $t$, and disjoint from $\cup_{j \in \{1, \ldots, M\} \backslash \{i\}} [t_j^0]$.
Use Urysohn's Lemma to find $x \in C_0(X)$ so that $0 \leq x \leq 1/\lambda_i$, $x = 0$ outside of $U$, and $\lambda_ix(t_i^k) = 1$. Fix $S$, and shows that
\begin{equation}
	\rho_s(x) = \lambda_p x( s_p^{\kappa_0} ) + 
	\sum_{r=1}^q \frac{x( s_{p+r}^{\kappa_r} )}{b ( s|^{r} )} \leq 1 + \vr
 \label{eq:desired_est}
\end{equation}
Once this is shown we shall conclude that $x$ satisfies $\tri{x} \leq 1 + \vr$, and therefore,
$$
\tri{ \lambda_i \delta_t}^*	\geq \frac{\lambda_i \delta_t(x)}{\tri{x}} = \frac1{1+\vr} .
$$
As $\vr$ can be arbitrarily small, we conclude that $\tri{ \lambda_i \delta_t}^*	\geq 1$.

To establish \eqref{eq:desired_est}, consider four cases.

(i) $p \geq M$. Recall that, by Lemma \ref{l:bmap}(\ref{est on b}), we have $b ( s|_{p}^{p+r}) \leq L^{4-3(p+r)}$. As the values of $x$ do not exceed $1/\lambda_i$, we have
$$
\rho_s(x) \leq \frac1{\lambda_i} \Big( \lambda_p + \sum_{r=1}^q L^{4-3(p+r)} \Big)
\leq \frac1{\lambda_i} \Big( \lambda_M + L^{5-3M} \Big) < 1 .
$$

(ii) $i < p < M$. By our choice of $x$ (and $U$), $x( s_{p+r}^{\kappa_r} ) = 0$ for $p+r < M$, hence
$$
\rho_s(x) \leq \frac1{\lambda_i} \sum_{r=M-p}^q L^{4-3(p+r)} < 1 .
$$

(iii) $p=i$. Again $x( s_{p+r}^{\kappa_r} ) = 0$ for $i < p+r < M$, so 
$$
\rho_s(x) \leq \frac1{\lambda_i} \Big( \lambda_i + \sum_{r=M-p}^q L^{4-3(p+r)} \Big) < 1 + L^{5-3M} < 1 + \vr.
$$

(iv) $1 \leq p < i$. Removing all vanishing terms, we obtain (invoking Lemma \ref{l:bmap}(\ref{est on b}) again)
$$
\rho_s(x) \leq \frac1{\lambda_i} \Big( \frac{x(s_i^{\kappa_{i-p}})}{b(s|^{i-p})} + \sum_{r=M-p}^q L^{4-3(p+r)} \Big) < \frac1{\lambda_i} \big( L^{-2} + L^{5-3M} \big) \leq 1 .
$$

Now consider $t \notin \cup_i [t_i^k]$.  To estimate $\tri{\delta_t}^*$ from above, fix $x \in C_0(X)$ and $\varepsilon > 0$. We have to show that $|x(t)| \leq \tri{x}+\vr$.
Find $\delta > 0$ so that $|x(s) - x(t)| < \vr$ whenever $d(s,t) < \delta$. Find $i, \gamma$ so that $d(t_i^\gamma,t) < \delta$. Then
$$
|x(t)| < \vr + \big| x \big( t_i^\gamma \big) | \leq \vr + \lambda_i \big| x \big( t_i^\gamma \big) | \leq \vr + \tri{x} ,
$$
as desired.

Estimating $\tri{\delta_t}^*$ from below proceeds as in the case of $t \in [t_i^0]$.
Just as before, fix $\varepsilon > 0$, and choose $M \in \N$ such that $\lambda_M + L^{5-3M} < 1 + \vr$. Find an open set $U$, containing $t$, which does not meet $\cup_{i \leq M} [t_i^0]$, and construct a function $x \in C_0(X)$ so that $0 \leq x \leq 1 = x(t)$, and $x$ vanishes outside of $U$. 
Clearly $\delta_t(x) = 1$. As in the earlier part of this proof, we show that $\rho_s(x) < 1 + \vr$ for any $s \in N$.
Consequently, $\tri{\delta_t}^* \geq x(t)/\tri{x} \geq 1/(1+\varepsilon)$.
\end{proof}

To characterize isometries on $(C_0(X), \tri{ \cdot })$, we shall need to evaluate norms of finite linear combinations of $\delta_t$'s. Before proceeding, we need to establish some geometric facts.
The first one is fairly straightforward.

\begin{lemma}\label{l:fin_dim}
	Suppose $Z$ is an $n$-dimensional Banach space, and there exist $z_0 \in Z$, $z_1^*, \ldots, z_n^* \in Z^*$ so that $\|z_0\| = 1$, and for any $i$, $\|z_i^*\|^* \leq 1$, and $z_i^*(z_0) = 1$. Then, for any $\alpha_1, \ldots, \alpha_n \geq 0$, we have $\sum_i \alpha_i = \| \sum_i \alpha_i z_i^*\| = (\sum_i \alpha_i z_i^*)(z_0)$.
\end{lemma}
Here, $\| \cdot \|$ and $\| \cdot \|^*$ denote the norm of $Z$ and the dual norm, respectively.

Below, we use the notation $\one$ for the vector $(1, \ldots, 1)$.

\begin{lemma}\label{l:near 1}
	Suppose $n \geq 2$, $1.1 \geq C > \lambda_1 > \ldots > \lambda_n > 1$, and the numbers $\zeta_{ij}$ ($1 \leq i < j \leq n$) satisfy $0 \leq \zeta_{ij} \leq 9^{4-3j}$. 
	For $1 \leq i \leq n$, let $z_i^* = (0, \ldots, 0, \lambda_i, \zeta_{i,i+1}, \ldots, \zeta_{in})$ (with $i-1$ zeros). 
	Let $\trm$ be the $n \times n$ matrix with rows $z_1^*, \ldots, z_n^*$ -- that is,
	$$
	\trm = \begin{pmatrix}  \lambda_1  &  \zeta_{12}  &  \zeta_{13}  &  \ldots  &  \zeta_{1n}  \\
		0  &   \lambda_2  &  \zeta_{23}  &  \ldots  &  \zeta_{2n}  \\
		\dots  &  \dots  &  \dots  &  \ldots  &  \dots  \\
		0  &   0  &  0  &  \ldots  &  \lambda_n
	\end{pmatrix}
	$$
	Let $z_0 = (z_{0k})_{k=1}^n$ be the unique solution of the equation $\mathcal T z = \one$. Then $\frac45 \leq z_{0k} \leq 1$ for any $k$.
	
	Suppose, furthermore, that $Z$ is the space $\R^n$, equipped with the norm $\| \cdot \|$, so that $\|z_0\| = z_i^*(z_0)$, and $\|z_i^*\|^* \leq 1$, for $1 \leq i \leq n$. 
	Then, if $z^* = (\beta_1, \ldots, \beta_n)$ satisfies $\vee_k |1 - \beta_k| \leq 1/5$, then $\|z^*\|^* = z^*(z_0)$.
\end{lemma}

Note that, in particular, we conclude that $\|z_k^*\|^* = 1$, for every $k$.

\begin{proof}
	The uniqueness of the solution $z_0$ follows from the fact that $A$ is non-singular. Using Gauss Elimination to find $z_0 = (z_{0k})_{k=1}^n$, we obtain
	\begin{equation}
		z_{0n} = \frac1{\lambda_n}, {\textrm{   and   }}
		z_{0k} = \frac1{\lambda_k} \big( 1 - \sum_{j=k+1}^n z_{0j} \zeta_{kj} \big) 
		{\textrm{   for   }}  1 \leq k \leq n-1.
		\label{eq:solve_for_z_n}
	\end{equation}
Clearly $z_{0k} < 1$ for any $k$. The lower bound for $z_{0k}$ follows from
$$
z_{0k} = \frac1{\lambda_k} \big( 1 - \sum_{j=k+1}^n z_{0j} \zeta_{kj} \big) \geq
\frac1{1.1} \big(1 - \sum_{j=k+1}^n 9^{4-3j} \big) > \frac1{1.1} \big(1 - \frac1{8 \cdot 9} \big) > \frac45 .
$$

	To address the ``furthermore' part of the lemma, note that $z_i^*(z_0) = 1$ for any $i$, hence $\|z_0\| = 1$. To apply Lemma \ref{l:fin_dim}, we shall show 
	that any $z^* \in U = \big[ \frac45,\frac65 \big]^n$ has a (necessarily unique) representation $z^* = \sum_i \alpha_i z_i^*$, with $\alpha_i \geq 0$.
	
	The vector $\alpha = (\alpha_i)$ is the solution to the equation $\mathcal T^T \alpha = z^*$ (here, $\mathcal T^T$ is the transpose of $\mathcal T$; note that the columns of $\mathcal T^T$ are $z_1^*, \ldots, z_n^*$).
	Write $z^* = (\beta_1, \ldots, \beta_n)$, then $\alpha_1 = \beta_1/\lambda_1$, and, for $2 \leq k \leq n$,
	$$
	\alpha_k = \frac1{\lambda_k} \big( \beta_k - \sum_{i=1}^{k-1} \alpha_i \zeta_{ik} \big) .
	$$
	It is easy to see that, if $|1 - \beta_k| < 1/5$ for any $k$, then $0 \leq \alpha_i < 2$ for any $i$.
	To complete the proof, apply Lemma \ref{l:fin_dim}.
\end{proof}

Our next goal is to determine the norm of certain linear combinations $\sum_{k=0}^n \beta_k \delta_{t_{i+k}^{\gamma_k}}$ ($t = \big(t_i^{\gamma_0}, \ldots, t_{i+n}^{\gamma_n}\big) \in N$). To this end, we introduce some notation. 

For $s = (s_0, \ldots, s_n) \in X^{n+1}$ ($n \geq 0$), denote by $I_s$ the ideal in $C_0(X)$, consisting of all $x \in C_0(X)$ which vanish on $s$.
Consider the $(n+1)$-dimensional space $Z_s = Z = C_0(X)/I_s$ (with the quotient norm). The natural quotient map $Q = Q_s : C_0(X) \to Z$ takes $x$ to $\big( x(s_k) \big)_{k=0}^n$. We shall therefore identify elements of $Z$ with sequences $z = (z_0, \ldots, z_n) \in \R^{n+1}$. We equip $Z$ with the quotient norm of $\tri{ \cdot }$, which we denote by $\| \cdot \|$. The dual norm shall be denoted by $\| \cdot \|^*$.

Fix $t = \big(t_i^{\gamma_0}, \ldots, t_{i+n}^{\gamma_n}\big) \in N$. Let $Z = Z_t$, and
\begin{equation}
z_k^* = \Big( 0, \ldots, 0, \lambda_{i+k} , \frac1{b(t|_{k}^{k+1})}, \ldots, \frac1{b(t|_{k}^n)} \Big) \, \, \, (0 \leq k \leq n) .
\label{eq:z-k-*}
\end{equation}

 Given $t,s\in N^{n+1}$, we say that $t$ and $s$ are \textit{almost equivalent} if $[t|_0^{n-1}] = [s|_0^{n-1}]$ and $[t|_{1}^n]=[s|_1^n]$.\\
 
 Suppose now that $t'= (t_p^{\kappa_0},...t_{p+q}^{\kappa_q}) \in N$ with $i\leq p<i+n$ and $1\leq q \leq i+n -p$.  We say that $t'$ is \textit{(p,q)-exceptional} (with respect to $t$) if 	\[ b(t_p^{\kappa_0},..., t_{p+q}^{\kappa_q} ) < b(t_p^{\gamma_{p-i}},...,t_{p+q}^{\gamma_{p-i+q}}) . \]
 	
Given $\varepsilon > 0$, we say that $t'$ is \textit{$(p,q, \varepsilon)$-exceptional} (with respect to $t$) if:
\begin{itemize}
\item 
$b(t_p^{\kappa_0},..., t_{p+q}^{\kappa_q} ) < b(t_p^{\gamma_{p-i}},...,t_{p+q}^{\gamma_{p-i+q}}) $ if either $i < p < p+q \leq i+n$, or $i \leq p < p+q < i+n$.
\item 
If $p=i$ and $q=n$, then $b(t_p^{\kappa_0},..., t_{p+q}^{\kappa_q}) <  \big( 1+ \varepsilon \big)^{-1}L^{c(t)}$.
\end{itemize} 

\begin{lemma}\label{l:pq-exceptional}
	 Let  $\varepsilon > 0$, fix $t = \big(t_i^{\gamma_0}, \ldots, t_{i+n-1}^{\gamma_{n-1}}, t_{i+n}^{\gamma_n} \big)\in N$, and suppose $s = (s_0,...,s_n)\in N$ be almost equivalent to $t$, while there is no $\gamma$ so that $\big[ \big(t_i^{\gamma_0}, \ldots, t_{i+n-1}^{\gamma_{n-1}}, t_{i+n}^\gamma \big) \big] = [s]$.  Then the following hold:

	   \begin{enumerate}
	   	\item There exist disjoint open sets $U_0,..., U_n$ with $t_{i+k}^{\gamma_k} \in U_k$ for $0\leq k\leq n$ so that for any $p,q$, and any $(p,q)$-exceptional $t' = (t_p^{\kappa_0},...t_{p+q}^{\kappa_q}) \in N$, there exists $j\in [0,q]$ such that $t_{p+j}^{\kappa_j} \notin U_{p+j-i}$.
	   	
	   	\item There exist disjoint open sets $U_0,..., U_n$ with $s_k \in U_k$ for $0\leq k\leq n$ so that for any $(p,q,\varepsilon)$-exceptional $t'\in N$,	then there exists $j \in [0,q] $ such that $t_{p+j}^{\kappa_j} \notin U_{p+j-i}$.
	   \end{enumerate} 
\end{lemma}

\begin{proof}
	Start with statement (1): Note that $b$ is invariant under the action of $G$. Due to the definition of $b$, if $i \leq p < i+n$ and $1 \leq q \leq i+n-p$, then there exist only finitely many orbit equivalence classes $[t']$, with $t' = \big(t_p^{ \kappa_0}, \ldots, t_{p+q}^{ \kappa_q} \big) \in N$,  $t'$ is $(p,q)$-exceptional with respect to $t$. \\
	
	The union of $(p,q)$-exceptional orbits $[t']$ is closed and nowhere dense in $X^{q+1}$, and disjoint from $[t|^{p-i}_{p-i+q}]$, hence we can find an open set which contains $t|^{p-i}_{p-i+q}$, but does not meet any $(p,q)$-exceptional orbits.
	Consequently, we can select open sets $U_0, \ldots, U_n$ so that, for any relevant $p$ and $q$,
	$U_p \times \ldots \times U_{p+q}$ does not meet any $(p,q)$-exceptional orbits. In terms of $b$, we have that, if
	$$
	b \big(t_p^{ \kappa_0}, \ldots, t_{p+q}^{\kappa_q} \big) < b \big( t_p^{\gamma_{p-i}}, \ldots, t_{p+q}^{\gamma_{p-i+q}} \big) ,
	$$
	then there exists $j \in [0,q]$ so that $t_{p+j}^{\kappa_j} \notin U_{p+j}$. 
	
	For statement (2), there likewise are only finitely many $(p,q,\varepsilon)$-exceptional orbits, and given that no $\gamma$ satisfies $\big[\big(t_i^{\gamma_0}, \ldots, t_{i+n-1}^{\gamma_{n-1}}, t_{i+n}^\gamma \big)\big] = [s]$, these orbits are also disjoint from $[s]$. From there, apply the same argument as in case 1.
\end{proof}

\begin{lemma}\label{l:norms_on_quotient}
	In the above notation, suppose $z = (u_k)_{k=0}^n \in Z$ is such that $\frac45 \leq u_k \leq 1$, for any $k$. Then $\|z\| = \max_{0 \leq k \leq n} z_k^*(z)$, with $z_k^*$ as in \eqref{eq:z-k-*}.
\end{lemma}

\begin{proof}
	One direction is straightforward: for any $z \in Z$ we have $\|z\| = \inf \{ \tri{x} : Qx = z \}$. By \eqref{eq:define_new_norm}, such an $x$ satisfies $\tri{x} \geq z_k^*(|z|)$, for any $k$; so, $\|z\| \geq \max_{0 \leq k \leq n} z_k^*(z)$.
	
	For the converse inequality, fix $\varepsilon > 0$. Fix $z = (u_0, \ldots, u_n)$, where
	$$\frac45 \leq \wedge_k u_k \leq \vee_k u_k = 1 .$$
	Let $A = \max_{0 \leq k \leq n} z_k^*(z)$ (clearly $1 \leq A \leq C$). 
	Our goal is to find $x \in C_0(X)$ so that $x(t_{i+k}^{\gamma_k}) = u_k$ for $1 \leq k \leq n$, and $\triple{x} \leq (1 + \varepsilon)A$.
	To achieve this, we are going to find ``sufficiently small'' disjoint open sets $U_k \ni t_{i+k}^{\gamma_k}$ ($0 \leq k \leq n$). Then we find a function $x$ so that
	$u_k = x(t_{i+k}^{\gamma_k}) \geq x(t) \geq 0$ for $t \in U_k$ ($0 \leq k \leq n$), and $x(t) = 0$ if $t \notin \cup_k U_k$. Our goal is to show that, for any $t' = \big( t_p^{\kappa_0}, \ldots, t_{p+q}^{\kappa_q} \big) \in N$, we have
	\begin{equation}
		\lambda_p x( t_p^{\kappa_0} ) + \sum_{r=1}^q \frac{x( t_{p+r}^{\kappa_r} )}{b ( t'|^r) } \leq (1 + \varepsilon) A .
		\label{eq:desired_inequality}
	\end{equation}
	
	We impose two restrictions on the sets $U_k$.
	
	\smallskip

	[R1] $\triangleright$
	Find $M > i+n$ so large that 
	$	\lambda_M + L^{3-M} < \min \{ \lambda_{i+n+1} , 1 + \vr \}$.
	Make sure that, for each $k$, $U_k \cap \big( \cup [t_j^0] : j \in [1,M] \backslash \{i+k\} \big) = \emptyset$. 
	
	\smallskip

	[R2] $\triangleright$ The sets $U_k$ satisfy Statement 1 of Lemma \ref{l:pq-exceptional}: if $t'$ is $(p,q)$-exceptional with respect to $t$, ten there exists $j\in [0,q]$ such that $t_{p+j}^{\kappa_j} \notin U_{p+j-i}$.

	\smallskip
	
	We shall establish \eqref{eq:desired_inequality} by considering several cases separately: (i) $p > M$; (ii) $p \in \{1, \ldots, i-1, i+n+1, \ldots, M\}$; (iii) $p= i+n$; (iv) $i \leq p \leq i+n$.
	The latter case will be the most difficult.
	
	For our computations, recall that
	$$
	\frac45 \leq \vee_{0 \leq k \leq n} u_k < \vee_{0 \leq k \leq n} \lambda_{i+k} u_k \leq A < C .
	$$
	Also, due to the restriction [R1], we have:
 \begin{equation}
     \label{eq:u-p as bound}
     x(t_{i+p}^{\kappa}) \leq u_p \, \, {\textrm{for any }} \, p \in \{0,1,\ldots,n\} \, \, {\textrm{and any }} \, \kappa .
 \end{equation}
	\begin{equation}
     \label{eq:x equals 0}
     x(t_{r}^{\kappa}) = 0 \, \, {\textrm{for any }} \, r \in \{1,\ldots,M\} \backslash \{i, \ldots, i+n\} \, \, {\textrm{and any }} \, \kappa .
 \end{equation}
 
	(i) $p > M$. Note first that, in this situation, $\lambda_p x( t_p^{\kappa_0} ) \leq A$.
	Indeed, if $t_p^{\kappa_0} \notin \cup_{k=0}^n U_k$, then $x( t_p^{\kappa_0} ) = 0$.
	If, on the other hand, $t_p^{\kappa_0} \in U_k$ for some $k$, then $x( t_p^{\kappa_0} ) \leq u_k$. 
	As the sequence $(\lambda_j)$ is decreasing, we have $\lambda_p \leq \lambda_{i+k}$, hence
	$\lambda_p x( t_p^{\kappa_0} ) \leq \vee_k \lambda_{i+k} u_k \leq A$. 
	Recall also that $x(t_{p+r}^{\kappa_r}) \leq 1$, and that $b ( t'|^r) > L^{p+r}$.	Therefore, the left hand side of \eqref{eq:desired_inequality} does not exceed  	$A + \sum_{s=M+1}^\infty L^{-s} < A + \varepsilon$,	as desired.
	
	(ii) If $p \in \{1, \ldots, i-1, i+n+1, \ldots, M\}$, the term $\lambda_p x( t_p^{\kappa_0} )$ in \eqref{eq:desired_inequality} vanishes by \eqref{eq:x equals 0}, hence the left hand side of that inequality does not exceed $\sum_{s=1}^\infty L^{-s} < 1$.
	
	(iii) If $p = i+n$, the only non-zero terms in the left hand side of \eqref{eq:desired_inequality} arise when $r = i+n$, or $r > M$. By \eqref{eq:u-p as bound}, this left hand side does not exceed
	$$
	\lambda_{n+i} u_n + \sum_{s=M+1}^\infty L^{-s} < A + \varepsilon \leq (1 + \varepsilon) A .
	$$
	
	(iv) The case of $p \in \{i, \ldots, i+n-1\}$ is the most intricate. Fix $p$, and write the left hand side of  \eqref{eq:desired_inequality} as
	$$
	\lambda_p x( t_p^{\kappa_0} ) + 
	\sum_{r=1}^{i+n-p} \frac{x( t_{p+r}^{\kappa_r} )}{b ( t'|^r) }  + \sum_{r=M+1-p}^q \frac{x( t_{p+r}^{\kappa_r} )}{b ( t'|^r) } .
	$$

	We have
	$$
	\sum_{r=M+1-p}^q \frac{x( t_{p+r}^{\kappa_r} )}{b ( t'|^r) } < \sum_{s=M+1 }^\infty L^{-s} < \varepsilon .
	$$
	It remains to show that
	\begin{equation}
		(*) := \lambda_p x( t_p^{\kappa_0} ) + 
		\sum_{r=1}^{i+n-p} \frac{x( t_{p+r}^{\kappa_r} )}{b ( t'|^r) } \leq A. 
		\label{eq:difficult_sum}
	\end{equation}
	Now note that if 
	\begin{equation}
		b \big( t'|^r\big) \geq b \big( t|_{p-i}^{p-i+r} \big) 
		\label{eq:ineq_for_b}
	\end{equation}
	for every $r \in [0,i+n-p]$, then
	$$
	(*) \leq \lambda_p u_{p-i} + 
	\sum_{r=1}^{i+n-p} \frac{u_{p-i+r}}{b ( t|_{p-i}^{p-i+r} )} \leq A . 
	$$
	Now suppose \eqref{eq:ineq_for_b} fails for some $r$ as above. 
	Find the smallest $r$ for which this happens. By the restriction [R2], there exists $j \in [0,r]$ so that $t_{p+j}^{\kappa_{p-i+j}} \notin U_{p-i+j}$. Consequently, $x(t_{p+j}^{\kappa_{p-i+j}}) = 0$.
	We shall establish \eqref{eq:difficult_sum} for $j > 0$; the case of $j = 0$ is treated similarly.
	In light of $x(t_{p+j}^{\kappa_{p-i+j}}) = 0$, and $x( t_{p+1}^{\kappa_r} ) \leq u_{p-i+r}$, we have
	$$
	(*) \leq \lambda_p u_{p-i} + 
	\sum_{r=1}^{j-1} \frac{u_{p-i+r}}{b ( t'|^r) } + 
	\sum_{r=j+1}^{i+n-p} \frac{u_{p-i+r}}{b ( t'|^r) } ,
	$$
	and therefore,
	$$
	A - (*) \geq \frac{u_{p-i+j}}{b ( t|_{p-i}^{p-i+j}) } - 
	\sum_{r=j+1}^{i+n-p} \frac{u_{p-i+r}}{b ( t'|^r) } .
	$$
	Note that, for $r \geq j+1$, By Lemma \ref{l:bmap}(\ref{b encodes} ),
	$b ( t'|^r) \geq L^{r-j} b ( t|_{p-i}^{p-i+j})$.
	Also, $\frac45 \leq u_k \leq 1$, and therefore,
	$$
	A - (*) \geq \frac1{b ( t|_{p-i}^{p-i+j} ) } \Big( \frac45 - \sum_{r=j+1}^\infty L^{j-r} \Big) > 0 ,
	$$
	which gives us \eqref{eq:difficult_sum}.
\end{proof}

Here and below, we use the ``dot product'' notation: if $a = (a_0, a_1, \ldots, a_n), b = (b_0, b_1, \ldots, b_n) \in \R^{n+1}$, then $a \cdot b = \sum_{i=0}^n a_i b_i$.


\begin{corollary}\label{c:norms on the dual}
	For any $t = \big(t_i^{\gamma_0}, \ldots, t_{i+n}^{\gamma_n}\big) \in N$, there exists $a(t) \in \R^{n+1}_+$ 
	so that, for any $\beta = (\beta_0, \ldots, \beta_n) \in U = [4/5,1]^{n+1}$, we have
	$$
	\tri{ \sum_{k=0}^n \beta_k \delta_{t_{i+k}^{\gamma_k}} } = a(t) \cdot \beta .
	$$
	Moreover, for $t, s \in N$, $a(t) = a(s)$ iff $[t] = [s]$.
\end{corollary}

\begin{proof}
	Let $Z = Z_t$ be as in Lemma \ref{l:norms_on_quotient}. By duality between quotients and subspaces, $(Z^*, \| \cdot \|^*)$ can be identified with $\spn \big[\delta_{t_{i+k}^{\gamma_k}} : 0 \leq k \leq n\big] \subset \big( C_0(X), \tri{ \cdot } \big)^*$.
	Thus, it suffices to show that, for $\beta = (\beta_0, \ldots, \beta_n)$ as in the statement of the lemma, $\|\beta\|^* = a(t) \cdot \beta$.
	
	For $0 \leq k \leq n$, consider the $n+1$-length vectors
	$$
	z_k^* = \bigg( \underbrace{\mbox{0,...,0}}_{{\tiny k}} \ , \lambda_{i+k}, \frac1{b(t|_k^{k+1})} , \ldots, \frac1{b(t|_k^{n})} \bigg) .
	$$
	By Lemma \ref{l:norms_on_quotient}, $\|z\| = \max_k z_k^*(z)$ whenever $z = (u_0, \ldots, u_n)$ has positive entries, and $\frac{\min_k u_k}{\max_k u_k} \geq \frac45$.
	Moreover, clearly $\|z_k^*\|^* \leq 1$.
	Let $\trm = \trm(t)$ be the matrix whose rows are vectors $z_k^*$, and let $z_0 = a(t)$ be the solution to $\trm z = \one$. The solution $a(t)$ is obtained by Gaussian elimination (see \eqref{eq:solve_for_z_n}, within Lemma \ref{l:near 1}), so the properties of the function $b$ tell us that $a(t) \neq a(s)$ unless $[t] = [s]$.
	By Lemma \ref{l:near 1}, $\|\beta\|^* = a(t) \cdot \beta$ if $\beta \in U$.
	%
\end{proof}


Combining Lemma \ref{l:limits of atoms} with Corollary \ref{c:norms on the dual}, we obtain:

\begin{lemma}\label{l:invariance}
	Suppose $t = \big(t_i^{\gamma_0}, \ldots, t_{i+n}^{\gamma_n}\big) \in N$ and $s = (s_0, \ldots, s_n) \in \widetilde{N}$ satisfies $[s] = [t]$.
	Then, for any $\beta = (\beta_0, \ldots, \beta_n) \in [4/5,1]^{n+1}$, we have
	$$
	\tri{ \sum_{k=0}^n \beta_k \delta_{s_k} } = a(s) \cdot \beta ,
	$$
	where the vector $a(s)$ equals $a(t)$, defined in Corollary \ref{c:norms on the dual}.
\end{lemma}

%

\begin{lemma}\label{l:comparison}
	Suppose $t = \big(t_i^{\gamma_0}, \ldots, t_{i+n}^{\gamma_n}\big) \in N$ and 
	$s = (s_0, \ldots, s_n) \in N$ are almost equivalent, and there is no $\gamma$ so that $\big[\big(t_i^{\gamma_0}, \ldots, t_{i+n-1}^{\gamma_{n-1}}, t_{i+n}^\gamma \big)\big] = \big[s\big]$.  If $\beta_0, \ldots, \beta_n \in \big[ \frac45, 1\big]$, then
	$$ \tri{ \sum_{k=0}^n \beta_k \delta_{s_k^{\gamma_k}} } = a(s) \cdot \beta , $$
	where $a(s)$ is the solution to the equation $\trm(s) a(s) = \one$, and
	$$
	\trm(s) = \begin{pmatrix}  \lambda_i  &  \zeta_{01}  &  \zeta_{02}  &  \ldots  &  \zeta_{0n}  \\
		0  &   \lambda_{i+1}  &  \zeta_{12}  &  \ldots  &  \zeta_{1n}  \\
		\dots  &  \dots  &  \dots  &  \ldots  &  \dots  \\
		0  &   0  &  0  &  \ldots  &  \lambda_{i+n}
	\end{pmatrix} .
	$$
	Here:
	\begin{itemize}
		\item $\zeta_{0k} = 1/b \big( s|_0^k)$ for $1 \leq k \leq n-1$.
		\item $\zeta_{jk} = 1/b \big( s|_j^k\big)$ for $j+1 \leq k \leq n$.
		\item 
		$\zeta_{0n} = L^{-c(t)}$. 
	\end{itemize}
\end{lemma}

\begin{proof}
	Denote by $z_k^*$ ($0 \leq k \leq n$) the rows of the matrix $\trm(s)$ described in the statement of the lemma.
	As in the proof of Corollary \ref{c:norms on the dual}, it suffices to show that $\|u\|_{Z_s} = \max_k z_k^* \cdot u$ whenever $u = (u_0, \ldots, u_n)$ satisfies $\vee_k |u_k - 1| \leq 1/5$, and also that $\max_{0 \leq k \leq n} \|z_k^*\|^* \leq 1$.

	
	First show that $\|u\|_{Z_s} \geq \max_k z_k^* \cdot |u|$ holds for any $u$ (this will imply that $\|z_k^*\|^* \leq 1$ for $0 \leq k \leq n$).
	It suffices to establish this inequality for $u \geq 0$, which is what we shall do.
	Suppose $x(s_k) = u_k$ for $0 \leq k \leq n$, and show that $\tri{x} \geq \max_k z_k^* \cdot u$.
	First consider $1 \leq k \leq n$.
	Find a sequence $(g_m) \subset G'$ so that $\lim_m g_m t_{i+k}^{\gamma_k} = s_k$ for $1 \leq k \leq n$. Due to the $G'$-invariance of the function $b$, we have
	$$
	\tri{x} \geq \lambda_{i+k} x \big( g_m t_{i+k}^{\gamma_k} \big) + \sum_{j=1}^{n-k} \frac{x(g_m t_{i+k+j}^{\gamma_{k+j}})}{b(t|_k^{k+j})} ,
	$$
	and the continuity of $x$ gives $\tri{x} \geq z_k^* \cdot u$.
	
	The case of $k=0$ requires more care.
	Find a new sequence $(g_m)_m \subset G'$ so that $\lim_m g_m t_{i+p}^{\gamma_p} = s_p$ for $0 \leq p \leq n-1$. 
	Then for no $\gamma$ can $s_n$ be a cluster point of $\big\{g_m t_{i+n}^\gamma\big\}_{m \in \N}$. Indeed, if for some $\gamma$ there was a subsequence of $\big(g_m t_{i+n}^\gamma\big)$ convergent to $s_n$, then we would have $[s] = \big[ t_i^{\gamma_0}, \ldots, t_{i+n-1}^{\gamma_{n-1}}, t_{i+n}^\gamma\big]$.
    Thus, passing to a subsequence of $(g_m)_m$ if necessary, we can assume that $s_n = \lim_m g_m t_{i+n}^{\kappa(m)}$, where $(\kappa(m))_m$ increases without a bound. From the definition of $\tri{x}$,
	$$
	\tri{x} \geq \lambda_i x \big( g_m t_i^{\gamma_0} \big) + \sum_{j=1}^{n-1} \frac{x(g_m t_{i+j}^{\gamma_j})}{b(t|^j)} 
 + \frac{x (g_m t_{i+n}^{\kappa(m)} )}{b(t_i^{\gamma_0}, \ldots, t_{i+n-1}^{\gamma_{n-1}}, t_{i+n}^{\kappa(m)})} .
	$$
	By the definition of $b$,
	$b \big(t_i^{\gamma_0}, \ldots, t_{i+n-1}^{\gamma_{n-1}}, t_{i+n}^{\kappa(m)}\big) \nearrow L^{c(t)}$, and therefore,
	$$
	\tri{x} \geq \lambda_i x \big( g_m t_i^{\gamma_0} \big) + \sum_{j=1}^{n-1} \frac{x(g_m t_{i+j}^{\gamma_j})}{b(t|^j)} 
 + \frac{x (g_m t_{i+n}^{\kappa(m)} )}{L^{c(t)}} .
	$$
	The continuity of $x$ gives $\tri{x} \geq z_0^* \cdot u$.

	Now fix $u = (u_0, \ldots, u_n)$ with $u_k \in [\frac45,1]$, and let $A = \max_{0 \leq k \leq n} z_k^* \cdot u$. To show that $\|u\|_{Z_s} \leq A$, fix $\vr \in (0,1/2)$, and find $x \in C(K)_+$ so that $x(s_k) = u_k$ for $0 \leq k \leq n$, and, for any $t' = \big( t_{p}^{\kappa_0},..., t_{p+q}^{\kappa_q} \big) \in N$, 
	\begin{equation}
		\rho_{t'}(x) = \lambda_p x( t_p^{\kappa_0} ) + \sum_{r=1}^q \frac{x( t_p^{\kappa_r} )}{b ( t'|^r) } \leq (1 + 2\varepsilon) A .
		\label{eq:new_desired_inequality}
	\end{equation}
 We shall normalize to $\vee_k u_k = 1$, then $A > 1$.
 
	As in Lemma \ref{l:norms_on_quotient}, we achieve this by constructing disjoint open sets $U_k \ni s_k$ ($0 \leq k \leq n$), and finding $x$ so that $0 \leq x \leq u_k$ on $U_k$, $x(s_k) = u_k$ (for all $k$), and $x = 0$ outside of $\cup_k U_k$. We impose two restrictions on the sets $U_k$.
	
	[R1'] $\triangleright$
	Find $M > i+n$ so large that $\lambda_M + L^{3-M} < \min\{ \lambda_{i+n+1}, 1 + \vr \}$.
	Make sure that, for each $k \in \{0, \ldots, n\}$, $U_k \cap \big( \cup [t_j^0] : j \in [1,M] \backslash \{i+k\} \big) = \emptyset$ (we can do this, since the sets $[t_j^0]$ are closed). $\triangleleft$
	
	[R2'] $\triangleright$ $U_0,...,U_n$ satisfies the requirements of Case (2) in Lemma \ref{l:pq-exceptional} for all $(p,q,\varepsilon)$-exceptional $t' \in N$. In particular, for any $(p,q,\varepsilon)$-exceptional $t'\in N$,	then there exists $j \in [0,q] $ such that $t_{p+j}^{\kappa_j} \notin U_{p+j-i}$.
	$\triangleleft$

	We shall establish \eqref{eq:new_desired_inequality} by considering several cases separately: (i) $p > M$; (ii) $p \in \{1, \ldots, i-1, i+n+1, \ldots, M\}$; (iii) $p= i+n$; (iv) $i \leq p < i+n$.
	The latter case will be the most difficult.
	
	For our computations, recall that
	$$
	\frac45 \leq \vee_{i \leq k \leq i+n} u_k < \vee_{i \leq k \leq i+n} \lambda_{i+k} u_k \leq A < C .
	$$
	Also, due to the restriction [R1'], $x(t_{i+p}^{\kappa}) \leq u_p$ for any $p \in [0,n]$ and any $\kappa$.
	
	(i) $p > M$. As the sequence $(\lambda_i)$ is decreasing, hence $\lambda_p x( t_p^{\kappa_0} ) \leq \vee_k \lambda_{i+k} u_k < 1$. Therefore, the left hand side of \eqref{eq:desired_inequality} does not exceed $A + \sum_{m=M+1}^\infty L^{-m} < (1 + \varepsilon)A$, as desired.
	
	(ii) $p \in \{1, \ldots, i-1, i+n+1, \ldots, M\}$. The term $\lambda_p x( t_p^{\kappa_0} )$ in \eqref{eq:new_desired_inequality} vanishes, hence the left side of that inequality does not exceed $\sum_{m=1}^\infty L^{-m} < 1 < A$.
	
	(iii) If $p = i+n$, then the only non-zero terms in the left hand side of \eqref{eq:new_desired_inequality} correspond to $r = i+n$, or $r > M$. Thus, this left hand side does not exceed
	$$
	\lambda_{n+i} u_n + \sum_{m=M+1}^\infty L^{-m} < A + \varepsilon \leq (1 + \varepsilon) A .
	$$
 
	(iv) As before, the case of $p \in \{i, \ldots, i+n-1\}$ is the hardest.
	Write the left hand side of  \eqref{eq:new_desired_inequality} as
	$$
	\lambda_p x( t_p^{\kappa_0} ) + 
	\sum_{r=1}^{i+n-p} \frac{x( t_p^{\kappa_r} )}{b ( t'|^r) }  + \sum_{r=M+1-p}^q \frac{x( t_p^{\kappa_r} )}{b ( t'|^r) } .
	$$
	We have
	$$
	\sum_{r=M+1}^q \frac{x( t_p^{\kappa_r} )}{b ( t'|^r) } < \sum_{m=M+1}^\infty L^{-m} < \varepsilon .
	$$
	It remains to show that
	\begin{equation}
		(*) := \lambda_p x( t_p^{\kappa_0} ) + 
		\sum_{r=1}^{i+n-p} \frac{x( t_{p+r}^{\kappa_r} )}{b ( t'|^r) } \leq (1+\varepsilon)A . 
		\label{eq:new_difficult_sum}
	\end{equation}

    Consider the equations
	\begin{equation}
		 b \big( t'|^r \big) \geq 
		   b \big( s|_{p-i}^{p-i+r} \big) 
		\label{eq:new_ineq_for_b}
	\end{equation}
 and 
	\begin{equation}
		(1+\varepsilon)b(t_i^{\kappa_0},..., t_{i+n}^{\kappa_n}) \geq L^{c(t)} . 
	\label{eq:new_ineq_for_b2}
	\end{equation}
 If \eqref{eq:new_ineq_for_b} holds for any $r$, and \eqref{eq:new_ineq_for_b2} holds,
	 then as in Lemma \ref{l:norms_on_quotient},
	\begin{align*}
		(*) &
		\leq
		\lambda_p u_{p-i} + 
		\sum_{r=1}^{i+n-p} \frac{u_{p-i+r}}{b ( t|_{p-i}^{p-i+r} )}
		\\
		&
		\leq 
		\lambda_p u_{p-i} + (1+\varepsilon)u_{i+n}\zeta_{0n}+ \sum_{r=1}^{i+n-p-1} u_{p-i+r} \zeta_{p-i,p-i+r}  
		\leq (1+2\varepsilon) A . 
	\end{align*}
	If \eqref{eq:new_ineq_for_b} fails for some $r$, find the smallest $r$ for which this happens.  By the restriction [R2'], there exists $j \in [0,r]$ so that $t_{p+j}^{\kappa_{p-i+j}} \notin U_{p-i+j}$.  Proceed as in the proof of Lemma \ref{l:norms_on_quotient}. 
	Finally, if \eqref{eq:new_ineq_for_b} holds but \eqref{eq:new_ineq_for_b2} fails,  restrictions [R1'] and [R2'] combined guarantee that $x(t_{n+i}^{\kappa_{n}}) = 0$, and $x( t_{p+1}^{\kappa_r} ) \leq u_{p-i+r}$ for $r < i-p+n$, so we have
	$$
	(*) \leq \lambda_p u_{p-i} + 
	\sum_{r=1}^{i+n-p-1} u_{p-i+r}\zeta_{p-i, p-i+r} \leq (1+\varepsilon) A.  \qedhere $$
\end{proof}

To complete the proof of Theorem \ref{t:main-theorem}, we need to show that any $T \in \isom(C_0(X), \tri{\cdot})$ is implemented by an element of $G$.
	Note that $T$ is a lattice isomorphism on $C_0(X)$, hence a weighted composition \cite[Theorem 3.2.10]{M-N}:
	$[Tx](k) = w(k) x(\phi k)$, where $w  \in C_b(X)_+$ is bounded away from $0$ and $\phi$ is a homeomorphism on $X$ (here $C_b(X)$ denotes the set of all bounded continuous functions on $X$).

	\begin{lemma}\label{l:w=1}
    The function $w$ defined above equals $1$ everywhere. Furthermore, for any $i$, $\phi [t_i^0] = [t_i^0]$.
	\end{lemma}

	\begin{proof}
	Observe that, for $k \in X$, $T^* \delta_k = w(k) \delta_{\phi k}$. Further, by Lemma \ref{l:unit},
	$\tri{\delta_k}^* = 1/\lambda_i$ for $k \in [t_i^0]$, while for $k\notin \cup_i [t_i^0]$, $\triple{\delta_k}^* = 1$. 
	The set $\cup_i [t_i^0]$ is meagre in $X$.
 Note that $\phi$, being a homeomorphism, takes nowhere dense sets to nowhere dense sets, hence meagre sets to meagre sets. Consequently, $X_0 = \big(\cup_i [t_i^0]\big) \cup \phi\big(\cup_i [t_i^0]\big)$ is meagre.
	
 For $k \notin X_0$, $\triple{\delta_k}^* = 1$, while $\triple{T^* \delta_k}^* = w(k) \tri{\delta_{\phi k}}^* = w(k)$.	This shows that $w = 1$ on $K_0$, and therefore, by the continuity of $w$, $w = 1$ everywhere.

	Now note that $\tri{\delta_k}^* = 1/\lambda_i$ iff $k \in [t_i^0]$. As $T^*$ is an isometry, we conclude that $k \in [t_i^0]$ iff $\phi k \in [t_i^0]$. This establishes the ``furthermore'' statement.
	\end{proof}
	
	\begin{lemma}\label{l:preserve_orbits}
	 For any $t\in \widetilde{N}$, $\phi(t) \in [t]$.
	\end{lemma}

	\begin{proof}
	We shall prove this statement by induction on $|t|$. The basic case ($|t| = 1$) is in Lemma \ref{l:w=1}. It remains to perform the inductive step: prove that our statement holds for $|t| =n+1$, provided it holds for $|t| = n$.
	
	Suppose, for the sake of contradiction, that $|t| = n+1$ is such that $\phi(t) \notin [t]$.
	By the induction hypothesis, $\phi(t|_0^{n-1}) \in [t|_0^{n-1}]$, and $\phi(t|_1^{n}) \in [t|_1^{n}]$.
	
	Write $t = \big( t_i^{\gamma_0}, \ldots, t_{i+n}^{\gamma_n} \big)$, and $\phi(t) = s = (s_0, \ldots, s_n)$. Then for any $\beta_0, \ldots, \beta_n$,
	\begin{equation}
	 \tri{ \sum_{k=0}^n \beta_k \delta_{t_{i+k}^{\gamma_k}} }^* = \tri{ \sum_{k=0}^n \beta_k \delta_{s_k} }^* .
	 \label{eq:cannot_hold}
	\end{equation}
	We shall show this cannot happen. Recall for future use that, whenever $\beta = \big(\beta_0, \ldots, \beta_n\big) \in \big[ \frac45, 1 \big]^{n+1}$, the left hand side of \eqref{eq:cannot_hold} is given by
	$$
	\tri{ \sum_{k=0}^n \beta_k \delta_{t_{i+k}^{\gamma_k}} }^* = a(t) \cdot \beta .
	$$
	Here, $a(t)$ is the solution to the equation $\trm(t) a(t) = \one$, where
	$$
 \trm(t) = \begin{pmatrix}  \lambda_i  &  \zeta_{01}(t)  &  \zeta_{02}(t)  &  \ldots  &  \zeta_{0n}(t)  \\
      0  &   \lambda_{i+1}  &  \zeta_{12}(t)  &  \ldots  &  \zeta_{1n}(t)  \\
      \dots  &  \dots  &  \dots  &  \ldots  &  \dots  \\
      0  &   0  &  0  &  \ldots  &  \lambda_{i+n}
     \end{pmatrix} ,
     {\textrm{   with   }}
        \zeta_{jk}(t) = \frac1{b\big( t|_{j}^{k} \big)} .
 $$
	
	
	To evaluate the right hand side of \eqref{eq:cannot_hold} for $\beta \in \big[ \frac45, 1 \big]^{n+1}$, consider two possible cases:
	
	(1) 
	$[s] = [t']$, where $t' = \big( t_i^{\gamma_0}, \ldots, t_{i+n-1}^{\gamma_{n-1}}, t_{i+n}^\kappa \big)$ for some $\kappa \neq \gamma_n$. Then Corollary \ref{c:norms on the dual} tells us that
	$\tri{ \sum_{k=0}^n \beta_k \delta_{s_k} }^* = a(t') \cdot \beta$
	whenever $\beta = \big(\beta_0, \ldots, \beta_n\big) \in \big[ \frac45, 1 \big]^{n+1}$.
	Here, $a(t')$ is the solution to the equation $\trm(t') a(t') = \one$; $\trm(t')$ is defined in a manner similar to the matrix $\trm(t)$, and has entries
	$$
	\zeta_{jk}(t') = \frac1{b\big( t'|_{j}^{k} \big)} .
	$$
	Recall that $t'|^{n-1}_0 \in [t|^{n-1}_0]$, and $t'|^n_{1} \in [t|^n_{1}]$, and therefore, $\zeta_{jk}(t') = \zeta_{jk}(t)$ unless $j = 0, k = n$.
	On the other hand $t' \notin [t]$, and therefore, $\zeta_{0n}(t') \neq \zeta_{0n}(t)$.
	Calculating $a(t)$ and $a(t')$ using Gauss elimination (with triangular matrices $\trm(t)$ and $\trm(t')$ respectively), we conclude that $a(t) \neq a(t') = a(s)$, hence \eqref{eq:cannot_hold} fails.

	(2) For no $t' = \big( t_i^{\gamma_0}, \ldots, t_{i+n-1}^{\gamma_{n-1}}, t_{i+n}^\kappa \big)$ we have $[s] = [t']$.
	By Lemma \ref{l:comparison}, for $\beta \in \big[ \frac45, 1 \big]^{n+1}$ we have
	$\tri{ \sum_{k=0}^n \beta_k \delta_{s_k} }^* = a(s) \cdot \beta$, where $\trm(s) a(s) = \one$ and the matrix $\trm(s)$ arises from entries
	$$
	\zeta_{jk}(s) = \zeta_{jk}(t) {\textrm{  for  }} (j,k) \neq (0,n) , {\textrm{  and   }}
	\zeta_{0n}(s) = L^{-c(t)} ,
	$$
	As noted in the proof of Lemma \ref{l:comparison}, $G' t_{i+n}^0$ is infinite, hence $\zeta_{0n}(t) < L^{-c(t)}$.
	Again, $a(t) \neq a(s)$, so \eqref{eq:cannot_hold} fails.
	\end{proof}
	
\begin{proof} [Proof of Theorem \ref{t:main-theorem}]
	By Lemma \ref{l:G_invariance}, we have $G\subseteq \isom(C_0(X), \triple{\cdot})$.  Now consider $T \in \isom(C_0(X), \triple{\cdot})$; by Lemma \ref{l:w=1}, it is implemented by a homeomorphism $\phi : X \to X$.
By Lemma \ref{l:preserve_orbits}, for any $t\in \widetilde{N}$, $\phi(t) \in [t]$.  
Taking $t = (t_1^0, \ldots, t_n^0)$, we conclude that for any $n \in \N$ there exists $g_n \in G$ such that $\max_{1 \leq k \leq n} d(\phi(t_k^0), g_n(t_k^0) \ ) < \frac{1}{n}$.  Recall that $(t_i^0)_{i \in \N}$ is dense in $X$, so $g_n$ converges to $\phi$ on a dense set.  Since $G$ is  locally equicontinuous, and $X$ is locally compact, $g_n \rightarrow \phi$ on all of $X$.  Since $G$ is SOT-closed, by Lemma \ref{l:equicontinuous+SOT=isometry}, $\phi \in G$.
\end{proof}

\section{Applications of Theorem \ref{t:main-theorem}}\label{s:more-results}

In this section use Theorem \ref{t:main-theorem} to obtain some more displayability results. 

\begin{proposition}\label{p:compact bounded groups}
Suppose $X$ is a locally compact Polish space, and $G$ be a subgroup of $\isom(C_0(X), \| \cdot \|)$, compact in the SOT, so that every orbit of $\{\phi_g : g \in G\}$ is nowhere dense in $X$. Then for all $C > 1$, there exists a renorming $\tri{ \cdot }$ on $C_0(X)$ such that $G = \isom(C_0(X), \tri{ \cdot })$ and $\| \cdot \| \leq \tri{ \cdot } \leq C \| \cdot \|$.
\end{proposition}

\begin{proof}
 By Proposition \ref{p:characterize-SOT-convergence}, the map $\isom(C_0(X)) \to {\mathrm{Hom}}(X) : g \mapsto \phi_g$ is continuous (here $\isom(C_0(X))$ and ${\mathrm{Hom}}(X)$ are equipped with the strong operator topology, and the topology of uniform convergence on compact sets, respectively).
 A continuous image of a compact set is compact, hence $\{\phi_g : g \in G\}$ is locally equicontinuous, by Ascoli-Arzela Theorem.
 Now invoke Theorem \ref{t:main-theorem}.
\end{proof}

 \begin{theorem}\label{t:expanded-theorem}
 Suppose $X$ is a locally compact Polish space and $G$ a relatively SOT-closed subgroup of $\isom(C_0(X))$
  such that $\{\phi_g : g \in G\}$ is locally equicontinuous for some complete metric on $X$. For any locally compact Polish space $L$ with no isolated points, $G$ is displayable on $C_0(X\times L)$.  
  
  In particular:
 	\begin{enumerate}
  \item
  $G$ is displayable on $C_0(X \times [0,1])$.
 	 \item 
 	If $X$ has no isolated points, $G$ is displayable on $C_0(X^2)$.
 	\item
 	If $X$ is compact, then $G$ is displayable on $C(X^\N)$. 
 	\end{enumerate}
 \end{theorem}

\begin{proof}
For $g \in G$ define the maps $\psi_g$, acting on the locally compact Polish space $X \times L$ via $\psi_g(k,\ell) = (\phi_g(k),\ell)$. The corresponding composition operator on $C_0(X \times L)$ shall be denoted by $\widetilde{g}$.
These operators form a group, which we denote by $\widetilde{G}$. By Proposition \ref{p:characterize-SOT-convergence}, the map $g \mapsto \widetilde{g}$ and its inverse are SOT-continuous group isomorphisms.

It is easy to observe that the homeomorphisms $\{\psi_g : g \in G\}$ act locally equicontinuously over $X \times L$.  Furthermore, since $L$ has no isolated points, the orbit of each point $(k,l)$ is nowhere dense. Theorem \ref{t:main-theorem} provides us with an equivalent norm $\tri{\cdot}$ on $C_0(X \times L)$ so that $G = \isom(C_0(X \times L), \triple{\cdot})$. 

The three ``moreover'' statements arise from a choice of $L$.
\end{proof}

\begin{remark}
Denote by ${\mathbb{T}}$ the one-dimensional torus; we can view it as a group of rotations.
	Theorem \ref{t:expanded-theorem} shows that $\mathbb T$ is displayable on the real Banach lattice $C({\mathbb{T}}^2)$. This contrasts with the Banach space situation: \cite[Theorem 29]{Fer11} shows that  ${\mathbb{T}}$ is not displayable on a real Banach space.   
\end{remark}

We also have the following: 

\begin{corollary}\label{c:everything displayed}
	Any locally compact Polish group $G$ is displayable on $C_0(X)$ for some locally compact Polish space $X$.  In addition, $X$ can be chosen such that any closed subgroup $H$ of $G$ is also displayable in $C_0(X)$.
\end{corollary}

\begin{proof}
	By \cite[Theorem 2.1]{Malicki09}, there exists a locally compact Polish space $Z$ such that $G$ is isomorphic to the group of surjective isometries of $Z$.
By Theorem \ref{t:expanded-theorem}, $G$, and all of its closed subgroups, are displayable on $C_0(Z \times [0,1])$.
\end{proof}

\section{Bounded groups of lattice isomorphisms}\label{s:isomorphisms}

In this section, we broaden our attention from groups of lattice isometries to groups of lattice isomorphisms (which again are assumed to be surjective).
We shall denote the group of lattice isomorphisms of a Banach lattice $Z$ by ${\mathrm{ISO}}^{\approx}(Z)$.
A group $G \subset {\mathrm{ISO}}^{\approx}(Z)$ is called \emph{bounded} if there is a constant $C$ such that for all $g\in G$ and $x\in Z$, $\frac{1}{C} \|x\| \leq \|gx \| \leq C \|x\|$.  If $G$ is bounded, 
then it acts isometrically on $(Z, \| \cdot \|_G)$, where \[ \|x\|_G = \sup_{g\in G} \|gx\| ; \] is  an equivalent norm on $Z$.
Note that $\| \cdot \|_G$ is a lattice norm. Moreover, if $\| \cdot \|$ is an AM-norm, then the same is true for $\| \cdot \|_G$. For more information on AM-norms on $C(K)$-spaces, see \cite{CDvR}.

Here we will examine displays of bounded groups of lattice isomorphisms on $C_0(X)$ ($X$ locally compact metrizable).
First recall from Lemma \ref{p:isomorphism-characterization} that each $g\in {\mathrm{ISO}}^{\approx}(C_0(X))$ is determined by a continuous weight function $a_g:X\rightarrow \R$ and underlying homeomorphism $\phi_g:X\rightarrow X$ such that for all $f\in C_0(X)$, $gf = a_g \cdot f\circ \phi_g$. 
For $g,h \in {\mathrm{ISO}}^{\approx}(C_0(X))$, we have $(hg)(f) = (a_h \cdot a_g\circ\phi_h ) \cdot f\circ (\phi_g \phi_h)$, so $\phi_{hg} = \phi_g\circ \phi_h$, and $a_{hg} = a_h \cdot a_g\circ\phi_h$.
In particular $\phi_{g^{-1}} = \phi_g^{-1}$, and $a_{g^{-1}} = 1/a_g\circ\phi_{g}^{-1}$.

 \begin{theorem}\label{t: bounded-group-weighted-norm}
 	Let $G$ be a bounded group of isomorphisms on $C_0(X)$.  Then there exists a (not necessarily continuous) function $m_G:X\rightarrow \R^+$ such that $\| x\|_G = \| m_G \cdot x\|_\infty $. 
 \end{theorem} 

\begin{proof}
	Let $m(t) = \inf_{g\in G} a_g(t)$. We now show that $\|x\|_G \leq 1$ iff $|x| \leq m$. Given $x \in C_0(X)$, we have:
	
	\begin{align*}
		\|x\|_G \leq 1 \iff & \forall g \forall t \ |(gx)(t) | \leq 1 \\
		\iff & \forall g \forall t \ |a_g(t) \cdot x(\phi_g(t)) | \leq 1 \\
		\iff &\forall g\forall t \ |x(\phi_g(t)) | \leq a_{g^{-1}}(\phi_g(t)) \\
		\iff &\forall t \forall g \ |x(t) | \leq a_{g^{-1}} (t) \\
		\iff &\forall t \ \ \ \ |x(t) | \leq m(t) \\
		\iff & |x| \leq m .
	\end{align*}

Since $G$ is bounded, $m$ is bounded below, so let $m_G = 1/m$.
\end{proof}

\begin{corollary}\label{nice renorming}
	Suppose the hypotheses of Theorem \ref{t: bounded-group-weighted-norm} are valid and, in addition, the family $\{a_g : g \in G\}$ is locally equicontinuous on $X$. Then $m_G$  is continuous. 
\end{corollary}

\begin{proof}
	  We will show that $m = 1/m_G$, defined in the proof for Theorem \ref{t: bounded-group-weighted-norm}, is continuous.  Since $m$ is bounded away from $0$, $m_G$ is also continuous. \\
	  Let $t\in X$ and $\varepsilon > 0$, and pick $\delta$ such that for all $s\in X$, $d(s,t) < \delta $ implies for all $g\in G, \ |a_g(s) - a_g(t) | < \varepsilon$. Then
   $$m(s)=\inf_{g\in G} | a_g(s) |  \leq \sup_{g\in G}  |a_g(s) - a_g(t) | + \inf_{g\in G} |a_g(t)| \leq \varepsilon + m(t). $$
	  Similarly, we have $m(t) \leq \varepsilon +m(s)$, so $|m(t) - m(s)| \leq \varepsilon$.
\end{proof}

\begin{remark}
The local equicontinuity of the family $(a_g)_{g \in G}$ cannot be omitted from the preceding proofs. For instance, let $X$ be the one-point compactification of $\{0,1\} \times \N$; that is, $X$ consists of points $(i,k)$ ($i \in \{0,1\}$ and $k \in \N$), as well as $\infty$.
    Equip $X$ with the following metric: $d \big( (i,k), (j,m) \big) = 2^{-\min\{k,m\}}$, $d \big( (i,k), \infty \big) = 2^{-k}$.

For $n \in \N$ define $a_n(t) = 2$ if $t = (0,n)$, $a_n(t) = 1/2$ if $t = (1,n)$, and $a_n(t) = 1$ otherwise. Further, let $\phi_n(i,n)= (1-i,n)$, and $\phi_n(t) = t$ for $t \notin \{(0,n),(1,n)\}$.
For $x \in C(X)$ let $g_n(x) = a_n \cdot x \circ \phi_n$. Note that $g_n$'s commute, and $g_n^2 = I$. One can show that $\{g_n : n \in \N\}$ is a bounded group of lattice isomorphisms on $C(X)$.

However, it is easily seen that the family $(a_n)$ is not equicontinuous (look at the vicinity of $\infty$). The discontinuity of $m$ is witnessed by $m(\infty) = 1$, while $m((1,n)) = 1/2$.
Also, the map $\Phi^{-1}$ is not SOT-continuous: $(g_n)$ is SOT-divergent (as can be seen by applying these maps to the constant one function), while $(\Phi g_n)$ converges in SOT to $I_{C(X)}$.
\end{remark}

 For $\phi \in C_b(X)$ (the space of bounded continuous functions on $X$), define the multiplication operator $\mult_\phi : C_0(X) \to C_0(X) : f \mapsto \phi f$.
 Define the map $\Phi$ on $B(C_0(X))$: for an operator $T$ on $C_0(X)$, let $\Phi T = \mult_{m_G} T \mult_{m_G^{-1}}$. Clearly $\Phi$ is a surjective algebraic isomorphism; also, both $\Phi$ and its inverse are SOT-continuous.

 Further, for any $x \in C_0(X)$ and $g \in G$, we have $\|m_G \cdot gx\|_\infty = \|m_G \cdot x\|_\infty$. Taking $y = m_G \cdot x$, we conclude that $\|m_G g(y/M_G)\|_\infty = \|y\|_\infty$. Thus, $\Phi(g)$ is an isometry on $C_0(X)$.

Summarizing the two preceding paragraphs, we obtain:

\begin{corollary}\label{handle groups}
Suppose $X$ is a locally compact metrizable space, and suppose $G$ is a bounded group of lattice isomorphisms of $C_0(X)$ for which the family $\{a_g : g \in G\}$ is locally equicontinuous on $X$.
Then there exists a group homomorphism $\Phi$, taking $G$ to a group $G'$ of lattice isometries of $C_0(X)$. Moreover, $\Phi$ is both norm and SOT continuous.
\end{corollary}

In light of Theorem \ref{t:main-theorem}, this immediately implies:

\begin{corollary}\label{display of bounded groups}
    Suppose $X$ is a locally compact Polish space, and suppose $G$ is a bounded relatively SOT-closed group of lattice isomorphisms of $C_0(X)$ for which the families $\{a_g : g \in G\}$ and $\{ \phi_g: g\in G\}$ are locally equicontinuous on $X$.  Suppose also that orbits in $X$ induced by $\{ \phi_g: g\in G\}$ are nowhere dense.    Then $G$ can be displayed on $C_0(X)$.
\end{corollary}

One may also ask what groups can be presented on $C_b(X)$ (the space of bounded continuous functions on the Hausdorff space $X$, equipped with the $\sup$ norm).
The following result shows that, at least for locally compact $X$, no extra groups arise.

 \begin{proposition}\label{p:isomorphisms of Cb}
  Suppose $X$ is a metric space.
Then any lattice isomorphism $T$ on $C_b(X)$ is of the form $[Tx](t) = a(t) x(\phi(t))$, where $a \in C_b(X)$ with $\inf a > 0$, and $\phi : X \to X$ is a homeomorphism. Moreover, $\|a\|_\infty = \|T\|$.
 \end{proposition}

 In particular, if $X$ is a locally compact metric space, then any $T$ as above maps $C_0(X)$ onto itself, and moreover, is completely determined by its action on $C_0(X)$.
 Consequently, if $X$ is a locally compact Polish space, then, by Corollary \ref{display of bounded groups}, any relatively SOT-closed bounded group of lattice isomorphisms of $C_b(X)$ satisfying the equicontinuity and density conditions mentioned in Corollary \ref{display of bounded groups}  can be displayed on $C_0(X)$.
 
 \begin{proof}
  We identify $C_b(X)$ with $C(\beta X)$, where $\beta X$ is the Stone-{\v{C}}ech compactification of $X$. For $x \in C_b(X)$, we denote by $x^\beta$ its unique extension to $\beta X$.
  Likewise, for $f \in C(\beta X)$, we shall use the notation $f_\beta$ for $f \circ \beta \in C_b(X)$.
  Then there exist $a \in C(\beta X)$ with $\inf a > 0$, and a homeomorphism $\phi : \beta X \to \beta X$ so that, for any $x \in C(\beta X)$ and $\omega \in \beta X$, we have $[Tx](\omega) = a(\omega) x(\phi(\omega))$. We have to show that $\phi^{-1}$ (and likewise, $\phi$) maps $X$ to itself.
  
  By \cite[Lemma 7.12]{GvR}, $\alpha \in \beta X$ belongs to $X$ iff there exists $x \in C(\beta X)$ so that $x$ vanishes at $\alpha$, and nowhere else. 
  For such $\alpha$ and $x$, $[Tx](\omega) = 0$ iff $\phi(\omega) = \alpha$, or in other words, $\omega = \phi^{-1}(\alpha)$.  Consequently, the function $Tx \in C(\beta X)$ vanishes at $\phi^{-1}(\alpha)$, and only there. Thus, $\phi^{-1}(\alpha) \in X$.
 \end{proof}

\section{Questions for further research}

\begin{question}
Can Theorem $\ref{t:main-theorem}$ be generalized to $AM$-spaces? \cite{OT_AM} provides a partial answer: any separable AM-space can be renormed to have trivial lattice isometries only.
In \cite{OT_AM}, a key step is to switch to a special kind of an AM-space called a \textit{regular AM space}. Thus, one approach would consist of displaying groups on regular AM spaces.
\end{question}

\begin{question}


 Can Theorem $\ref{t:main-theorem}$ be extended to the spaces $C_0(X)$, where $X$ has isolated points? In particular, what happens when $X = \N$ (and $C_0(X)$ is simply $c_0$)?
 Note that, by \cite{Fer11}, if $G$ is a closed subgroup of the infinite permutation group $S_\infty$, then $G \times \{-1,1\}$ can be displayed on $S_\infty$, in the Banach space sense. However, the new norm on $c_0$ producing this display is not a lattice norm.
\end{question}

\begin{question}
	Is every Polish group displayable in a separable lattice for some lattice $X$? 

To motivate this question, recall that, by \cite[Theorem 31]{Fer11}, for any Polish group $G$ there exists a separable real Banach space $X$ so that $G \times \{-1,1\}$ is displayable on $X$. The $\{-1,1\}$ multiple is important here, as any group displayable on a real Banach space must include a central involution.
\end{question}


\begin{thebibliography}{22}
	
	
	\bibitem{Be} S. Bellenot. Banach spaces with trivial isometries.
	Israel J. Math. 56 (1986), 89--96. 

	
	\bibitem{CDvR} I.~Claus, D.~Dondergoor, and A.~van Rooij.
	$M$-seminorms on spaces of continuous functions. Indag. Math. (N.S.) 11 (2000), 539--546.


\bibitem{Engelking} R. Engelking. General topology.
Heldermann Verlag, Berlin, 1989. 

	\bibitem{FeG_TAMS10} V. Ferenczi and E.M. Galego.
	Countable groups of isometries on Banach spaces.
	Trans. Amer. Math. Soc. 362 (2010), 4385--4431. 
	
	\bibitem{GvR} G.~Groenewegen and A.~van Rooij.
	Spaces of continuous functions. 	Atlantis Press, 2016.
	
	\bibitem{Fer11} V. Ferenczi and C. Rosendal. Displaying Polish Groups on Separable
	Banach Spaces. Extracta Math. 26 (2011), 195--233 (2011)
	
	\bibitem{Jar} K.~Jarosz.
	Any Banach space has an equivalent norm with trivial isometries.
    Israel J. Math. 64 (1988), 49--56. 

\bibitem{Kechris} A. Kechris. Classical descriptive set theory. 
Springer-Verlag, New York, 1995. 

	\bibitem{Malicki09} M. Malicki and S. Solecki. 
	Isometry groups of separable metric spaces. 
	 Math. Proc. Cambridge Philos. Soc. 146 (2009), 67--81. 
	 
	\bibitem{M-N} P. Meyer-Nieberg. Banach lattices. Springer-Verlag, Berlin, 1991.

	\bibitem{OT_AM} T. Oikhberg and M.A. Tursi. Renorming AM-spaces. Proc. Amer. Math. Soc. 150 (2022), 1127--1139.



	
\end{thebibliography}
\end{document}